\documentclass[12pt]{article}

\usepackage{geometry}
\usepackage{hyperref} 

\usepackage{amsmath}
\usepackage{amssymb}
\usepackage{amscd}
\usepackage{textcomp}
\usepackage{dsfont}
\usepackage{mathrsfs}

\long\def\forget#1{}

\newcommand{\Verkuerzung}[2]{#1}

\newcommand{\lang}[1]{\mbox{#1}}

\usepackage{color}
\newcounter{commentcounter}

\newcounter{urscommentcounter}

\def\?{\ 
{\bf\color{red}???}\ 
\immediate\write16{}
\immediate\write16{Warning: There was still a question mark . . . }
\immediate\write16{}}

\usepackage{amsthm}
\theoremstyle{plain}
\newtheorem{theorem}{Theorem}[section]

\newtheorem{lemma}[theorem]{Lemma}

\newtheorem{corollary}[theorem]{Corollary}
\newtheorem{proposition}[theorem]{Proposition}

\theoremstyle{definition}
\newtheorem{definition}[theorem]{Definition}
\newtheorem{definition-theorem}[theorem]{Definition-Theorem}
\newtheorem{definition-remark}[theorem]{Definition-Remark}
\newtheorem{question}[theorem]{Question}

\newtheorem{example}[theorem]{Example}
\newtheorem{remark}[theorem]{Remark}

\theoremstyle{remark}

\usepackage{xy}
\xyoption{all}

\newdir^{ (}{{}*!/-3pt/\dir^{(}}    
\newdir^{  }{{}*!/-3pt/\dir^{}}    
\newdir_{ (}{{}*!/-3pt/\dir_{(}}    
\newdir_{  }{{}*!/-3pt/\dir_{}}

\newcounter{zahl}

\def\theenumi{(\alph{enumi})}

\def\p@enumii{\theenumi}

\newcommand{\DS}{\displaystyle}
\newcommand{\TS}{\textstyle}
\newcommand{\SC}{\scriptstyle}
\newcommand{\SSC}{\scriptscriptstyle}

\newcommand{\cC}{\mathcal{C}}

\newcommand{\cG}{\mathcal{G}}

\DeclareMathOperator{\Aut}{Aut}

\DeclareMathOperator{\End}{End}

\DeclareMathOperator{\Gal}{Gal}
\DeclareMathOperator{\GL}{GL}

\DeclareMathOperator{\Koh}{H}

\DeclareMathOperator{\Hom}{Hom}

\DeclareMathOperator{\Quot}{Frac}

\DeclareMathOperator{\SL}{SL}

\DeclareMathOperator{\Spec}{Spec}
\DeclareMathOperator{\Spf}{Spf}

\DeclareMathOperator{\Tr}{Tr}

\newcommand{\alg}{{\rm alg}}

\DeclareMathOperator{\diag}{diag}

\DeclareMathOperator{\equi}{equi}

\newcommand{\fpqc}{{\it fpqc\/}}

\DeclareMathOperator{\id}{\,id}

\DeclareMathOperator{\rk}{rk}
\newcommand{\sep}{{\rm sep}}

\DeclareMathOperator{\whtimes}{\mathchoice
            {\wh{\raisebox{0ex}[0ex]{$\DS\times$}}}
            {\wh{\raisebox{0ex}[0ex]{$\TS\times$}}}
            {\wh{\raisebox{0ex}[0ex]{$\SC\times$}}}
            {\wh{\raisebox{0ex}[0ex]{$\SSC\times$}}}}

\renewcommand{\phi}{\varphi}
\renewcommand{\epsilon}{\varepsilon}

\usepackage{amsfonts}
\newcommand{\BOne} {{\mathchoice{\hbox{\rm1\kern-2.7pt l\kern.9pt}}
                              {\hbox{\rm1\kern-2.7pt l\kern.9pt}}
                              {\hbox{\scriptsize\rm1\kern-2.3pt l\kern.4pt}}
                              {\hbox{\scriptsize\rm1\kern-2.4pt l\kern.5pt}}}}

\newcommand{\BA}{{\mathbb{A}}}

\newcommand{\BD}{{\mathbb{D}}}

\newcommand{\BF}{{\mathbb{F}}}
\newcommand{\BG}{{\mathbb{G}}}

\newcommand{\BL}{{\mathbb{L}}}

\newcommand{\BN}{{\mathbb{N}}}

\newcommand{\BP}{{\mathbb{P}}}
\newcommand{\BQ}{{\mathbb{Q}}}
\newcommand{\BR}{{\mathbb{R}}}
\newcommand{\BS}{{\mathbb{S}}}

\newcommand{\BZ}{{\mathbb{Z}}}

\newcommand{\CA}{{\cal{A}}}
\newcommand{\CB}{{\cal{B}}}
\newcommand{\CC}{{\cal{C}}}

\newcommand{\CE}{{\cal{E}}}
\newcommand{\CF}{{\cal{F}}}
\newcommand{\CG}{{\cal{G}}}

\newcommand{\CI}{{\cal{I}}}

\newcommand{\CL}{{\cal{L}}}

\newcommand{\CN}{{\cal{N}}}
\newcommand{\CO}{{\cal{O}}}

\newcommand{\CS}{{\cal{S}}}
\newcommand{\CT}{{\cal{T}}}
\newcommand{\CU}{{\cal{U}}}
\newcommand{\CV}{{\cal{V}}}

\newcommand{\CX}{{\cal{X}}}
\newcommand{\CY}{{\cal{Y}}}
\newcommand{\CZ}{{\cal{Z}}}

\newcommand{\FG}{{\mathfrak{G}}}

\newcommand{\FL}{{\mathfrak{L}}}

\newcommand{\scrH}{{\mathscr{H}}}

\let\setminus\smallsetminus

\newcommand{\es}{\enspace}

\newcommand{\ul}[1]{{\underline{#1}}}
\newcommand{\ol}[1]{{\overline{#1}}}
\newcommand{\wh}[1]{{\widehat{#1}}}
\newcommand{\wt}[1]{{\widetilde{#1}}}

\DeclareMathOperator{\Nilp}{\CN \!{\it ilp}}

\usepackage{ifthen}

\newcommand{\invlim}[1][]{\ifthenelse{\equal{#1}{}}
{\DS \lim_{\longleftarrow}}
{\DS \lim_{\underset{#1}{\longleftarrow}}}
}

\newcommand{\dirlim}[1][]{\ifthenelse{\equal{#1}{}}
{\DS \lim_{\longrightarrow}}
{\DS \lim_{\underset{#1}{\longrightarrow}}}
}

\newcommand{\dbl}{{\mathchoice{\mbox{\rm [\hspace{-0.15em}[}}
                              {\mbox{\rm [\hspace{-0.15em}[}}
                              {\mbox{\scriptsize\rm [\hspace{-0.15em}[}}
                              {\mbox{\tiny\rm [\hspace{-0.15em}[}}}}
\newcommand{\dbr}{{\mathchoice{\mbox{\rm ]\hspace{-0.15em}]}}
                              {\mbox{\rm ]\hspace{-0.15em}]}}
                              {\mbox{\scriptsize\rm ]\hspace{-0.15em}]}}
                              {\mbox{\tiny\rm ]\hspace{-0.15em}]}}}}
\newcommand{\dpl}{{\mathchoice{\mbox{\rm (\hspace{-0.15em}(}}
                              {\mbox{\rm (\hspace{-0.15em}(}}
                              {\mbox{\scriptsize\rm (\hspace{-0.15em}(}}
                              {\mbox{\tiny\rm (\hspace{-0.15em}(}}}}
\newcommand{\dpr}{{\mathchoice{\mbox{\rm )\hspace{-0.15em})}}
                              {\mbox{\rm )\hspace{-0.15em})}}
                              {\mbox{\scriptsize\rm )\hspace{-0.15em})}}
                              {\mbox{\tiny\rm )\hspace{-0.15em})}}}}

\newcommand{\dotBD}{\vbox{\hbox{\kern2pt\bf.}\vskip-4.5pt\hbox{$\BD$}}}

\def\longto{\longrightarrow}
\def\into{\hookrightarrow}

\def\isoto{\stackrel{}{\mbox{\hspace{1mm}\raisebox{+1.4mm}{$\SC\sim$}\hspace{-3.5mm}$\longrightarrow$}}}

\newbox\mybox
\def\arrover#1{\mathrel{
       \setbox\mybox=\hbox spread 1.4em{\hfil$\scriptstyle#1$\hfil}
       \vbox{\offinterlineskip\copy\mybox
             \hbox to\wd\mybox{\rightarrowfill}}}}

\newcommand{\BaseOfD}{\BF}

\newcommand{\genericG}{P}

\newcommand{\Sht}{Sht}

\DeclareMathOperator{\SpaceFl}{\CF\ell}
\newcommand{\tauGlob}{\tau}
\newcommand{\tauLoc}{\hat\tau}
\newcommand{\charsect}{s}

\begin{document}

\author{Esmail Arasteh Rad\forget{\footnote{Part of this research was carried out while I was visiting Institute For Research In Fundamental Sciences (IPM).}} and Somayeh Habibi}

\date{\today}

\title{Some Motivic Remarks On The Moduli Stacks Of global $\FG$-Shtukas And Their Local Models\\}

\maketitle

\begin{abstract}

In this article we study motives  corresponding to the moduli stacks of $\FG$-shtukas and their local models. In particular we deal with the question of describing their motivic fundamental invariants.  As an application, we provide a criterion for mixed Tateness of the local model and discuss the semi-simplicity of Frobenius on their cohomology. We then use the theory of local models to reformulate a purity result for these moduli stacks in the motivic context.

\end{abstract}

~~~~~~~~~~~~~~~~~~~~~~~~~~~~~~~~~~~~~~~~~~~~~~~~~~~~~~~~~~~~~~
`To Urs Hartl for his 50th Birthday'

\section*{Introduction}

In \cite{AH_LM} we developed the theory of local models for moduli stacks of global $\FG$-shtukas beyond the constant split reductive case. Namely, we considered the case where $\FG$ is only a smooth affine group scheme over a curve $C/\BF_q$.

Recall that a Shimura data $(\BG,X,K)$ consists of a reductive group $\BG$ over $\BQ$, $\BG(\BR)$-conjugacy class $X$ of homomorphisms $\BS\to G_\BR$ for the Deligne torus $\BS$ and a compact open sub-group $K\subseteq \BG(\BA_f)$, subject to certain conditions. To such data one associate a Shimura variety; e.g. see \cite{MilneShimura}. In the analogous picture over function fields, the Shimura data replaces by a tuple $(\FG, \ul{\hat{Z}}, H)$, which is called \emph{$\nabla\scrH$-data}. These data consist of a smooth affine group scheme $\FG$ over a smooth projective curve $C$ over $\BF_q$, an $n$-tuple of (local) bounds $\ul{\hat{Z}}:=(\hat{Z}_{\nu_i})_{i=1\dots n}$, at the fixed characteristic places $\nu_i\in C$ and a compact open subgroup $H\subseteq \FG(\BA_C^\ul\nu)$; see \cite{Ar18} for further details on this analogy.  \\
To a $\nabla\scrH$-data $(\FG, \ul{\hat{Z}}, H)$ we associate a moduli stack
$\nabla_n^{H,\ul{\hat{Z}}}\scrH^1(C,\FG)^\ul\nu$ parametrizing global $\FG$-shtukas with level $H$-structure which are bounded by $\ul{\hat{Z}}$. Recall further that as part of these data, in \cite{AH_Local}, the first named author and Hartl have introduced notion of bounedness condition $\hat{Z}$, which generalizes the previous notions used by Drinfeld, Lafforgue, Varshavsky and others to the case where $\FG$ is not necessarily a constant reductive group over $C$, e.g. it may ramify at certain places. This roughly consists of (certain classes of) sub-schemes $\hat{Z}$ of $\wh{\CF\ell}_\BP:=\CF\ell_\BP\wh\times\BF_q\dbl\zeta\dbr$. Here $\CF\ell_\BP$ denote the partial affine flag variety associated to the group $\BP$ over $\Spec \BF\dbl z \dbr$. In \cite{AH_LM} we generalized this notion to the global situation. This is done by replacing $\hat{\ul Z}$  by (certain classes of) sub-schemes $\CZ$ of Beilinson-Drifeld affine Grassannian $GR_n$, which are subject to some conditions. We further  proved that both sub-schemes (corresponding to the above boundedness conditions) maybe regarded as local models for the moduli stacks of global $\FG$-shtukas. Namely, (a product of) the first type bounds may appear as the analog of \emph{Rapoport-Zink type local model} for Shimura varieties, in the sense of \cite{RZ}, and the second one maybe regarded as \emph{Beilinson-Drinfeld-Gaitsgory-Varshavsky type local model}; compare also Pappas and Zhu \cite{PZ}. For details in this direction see \cite[Sections 3.2 and 4.4]{AH_LM} and \cite{Ar18}.\\

In the present article we first study the motives associated to the local models for the moduli of $G$-shtukas. We are particularly interested on possible descriptions of the motivic fundamental invariants (in the sense of Huber and Kahn \cite{H-K}, see also definition \ref{DefSliceFilt} for the definition of the n'th motivic invariant $c_n(M)$ of a motive $M$) of the special fiber of the local models. As we will see below, this question is tightly related to the expectation that Frobenius acts semi-simply on the cohomology of these moduli stacks. Furthermore, we implement the theory of local models to transmit some results on the motivic intersection cohomology of the local models to the moduli stacks of $\FG$-shtukas. Along the way we prove some further miscellaneous results. We then discuss the relation between the local and global situation via the degeneration theory. \\ 

Let us briefly review the content of this article. 
In section \ref{SectPreliminaries} we recall some necessary background preliminaries.  In subsection \ref{SubSectSSFrobLM}, we consider the local situation (which corresponds to Rapoport-Zink local model). Recall that from $\BA^1$-homotopy theory point of view, one may contract affine subspaces inside a variety. This naturally suggests to consider those varieties that can be constructed as a tower of cellular fibrations, and then proceed towards more delicate cases by relating them to this particular case using tools such as decomposition theorem (in the sense of \cite{BBD} and also in the motivic sense of \cite{CoHa} and \cite{C-M}), Leray-Hirsch theorem and etc. Note that there are certain restrictions imposed by implementing decomposition theorem and Leray-Hirsch theorem in the motivic set up, which are in fact arising from difficult problems related to properties of cycle class map. On the other hand, known properties of the resolution of singularities for Schubert varieties inside affine flag varieties suggest to restrict our attention to certain class of boundedness conditions, that admit nice resolution of singularities; see definition \ref{DefCatNett} and parts \ref{cnNett} and \ref{semi-smallres} of the following theorem. This approach basically uses the machinery of slice filtration \cite{H-K}, together with results of \cite{CoHa} and \cite{C-M}, and the authors previous work \cite{AH17}. Then we use the geometric relation between local and global boundedness conditions and the degeneration method to treat the global case.


\begin{theorem}
Let $\hat{Z}$ be a boundedness condition in the sense of definition \ref{DefBDLocal} and let $Z$ denote its special fiber; see remark \ref{RemSpecialFiber}. We have the following statements

\begin{enumerate}

\item\label{cnNett}
Assume that $\hat{Z}$ is $c_{\leq n}$-nett in the sense of Definition \ref{DefNett}. Then the $n$-th motivic fundamental invariant $c_n(M(Z))$ lies in $\textbf{D}_f^b(\textbf{Ab})$.\\

\item\label{semi-smallres}
Assume that  $\ol Z:=Z\times_k \ol k$ is irreducible and admits a stratified $\ol Z=\coprod_{\beta\in \CB} \ol Z_\beta$ semi-small resolution of singularities $\ol\Sigma\to \ol Z$, with $c_n(M(\ol \Sigma))\in \textbf{D}_f^b(\textbf{Ab})$.  Then $c_n(M(\ol Z))\in \textbf{D}_f^b(\textbf{Ab})\subseteq \textbf{DM}_{gm}^{eff}(\ol k)$.\\

Moreover:\\

\item\label{nettisMT}
Assume that $\hat{Z}$ is nett (resp. there is  a semi-small resolution $\ol\Sigma\to \ol Z$ with $M(\ol\Sigma)\in MT\textbf{DM}_{gm}^{eff}(\ol k)$). Then $M(Z)$ (resp. $M(\ol Z)$) lies in $MT\textbf{DM}_{gm}^{eff}(k)$ (resp. $MT\textbf{DM}_{gm}^{eff}(\ol k)$). \\

\item\label{semismallresisMT}
In either of the cases mentioned in c), frobenius acts semi-simply on the cohomology $\Koh^i(\hat{Z},\BQ_\ell)$.

\end{enumerate}

\end{theorem}
\noindent
This is theorem \ref{ThmFrobistSemiSimpleOnZ} in the text. Note that the assumptions in part (c) are fulfilled when $\FG$ is constant $\FG:=G\times_{\BF_q} C$ with $G$ split reductive. 

In subsection \ref{SubSectModuliOfG-Sht} we briefly discuss the construction of the motives corresponding to the moduli stacks of $\FG$-shtukas inside Voevodsky's motivic categories. We discuss the category of motives over these moduli stack, we further explain how the local model theory can be used to transmit the above observations to the moduli stacks of global $\FG$-shtukas. Namely, in proposition \ref{Prop_MotIntersectionCompNablaScrH} we observe that the intersection motive $ICM(\nabla_n^{H,\ul{\hat{Z}}}\scrH^1(C,\FG)_s^\ul\nu)$ of the special fiber of $\nabla_n^{H,\ul{\hat{Z}}}\scrH^1(C,\FG)^\ul\nu$ coincides the restriction of $ICM(Hecke_n^{\ul{\hat{Z}}}(C,\FG)_s)$ up to some shift and Tate twist. Here $Hecke_n^{\ul{\hat{Z}}}(C,\FG)$ denotes the Hecke stack, see definition \ref{Def_Hecke} and definition-remark \ref{Def-RemHecke}. Recall that one has the following nice description of the geometry of the algebraic stack $Hecke_n^{\CZ}(C,\FG)$ as a family over $\scrH^1(C,\FG)\times C^n$. Namely, $Hecke_n(C,\FG)$ (resp. $Hecke_n^\CZ(C,\FG)$) and the fiber product $GR_n\times \scrH^1(C,\FG)$ (resp. $\CZ\times\scrH^1(C,\FG)$) are locally isomorphic as families over $\scrH^1(C,\FG)\times C^n$; see definition-remark \ref{Def-RemHecke}\ref{heckeislocallGr}. This subsequently arise the question that up to what extent one can describe the motive of the local model $\CZ$? Finally in subsection \ref{SubSecLocToGlob}, using Voison's degeneration method, we discuss the global situation. As a feature, this method avoids the complications related to analyzing the properties of the resolution of singularities. We in particular show

\begin{theorem}
Let $\FG$ be a parhoric group scheme over $C$. We have the following statements
\begin{enumerate}
\forget{\item
There is an assignment 
$$
\gamma: \{\text{global boundedness conditions}~ \CZ\} \to \{\text{n-tuple of local boundedness conditions}~ \ul{\hat{Z}}\}.
$$

Moreover a tuple $\hat{\ul Z}$ may give rise to a global boundedness condition $\CZ$.
}

\item
Assume that the global boundedness condition $\CZ$ arises from n-tuple $\hat{\ul Z}$ of local boundedness conditions $\hat{Z}_i$. 
Assume further that for every $i$ the motive of the generic fiber of $\hat{Z}_i$ lies in the category of pure Tate motives, then all fibers of a representative of the boundedness condition $\CZ$ over $\wt C^n$  are geometrically pure Tate. Here $\wt C$ denotes the corresponding reflex curve.

\item

Let $(\mu_i)$ be an n-tuple of cocharacters of $\FG$. Assume further that the Schubert varieties $S(\mu_i)$ are smooth and pure Tate. Then the motive $M(\CZ)$ of $\CZ:=\CZ(\mu_1,\dots,\mu_n)$ lies in the thick subcategory of $DM_{gm}(k)$ generated by pure Tate motives and $M(\wt C)$. In particular when $\FG$ is constant and $C=\BP^1$, the motive $M(\CZ)$ lies in the category $MT\textbf{DM}_{gm}^{eff}(k)$ of mixed Tate motives.

\forget{
\item Assume that $\FG$ is constant and let $\CZ$ be as in the above item (c). The class of the motive $[M(\CZ)]$ in the Grothendieck ring $K_0[DM_{-}^{eff}(k)]$ can be written as the following sum

$$
[M(\CZ)]= \sum_{\ul\alpha} [M^c((C^n)_\ul\alpha)]\cdot[M(S(\mu_\ul\alpha)]\cdot \prod_{i\in\ul\alpha} [M(S(\mu_i))].
$$

Here $\ul\alpha$ runs over subsets of $\{1,\dots,n\}$ and
$\mu_\ul\alpha:=\sum_{\alpha\in\ul\alpha}\mu_\alpha$\forget{, $S^\ul\alpha:=\prod_{\alpha\notin\ul\alpha} S(\mu_\alpha)$}. When $C=\BP^1$ the above class lies in the subring generated by $MT\textbf{DM}_{gm}^{eff}(k)$.  
}
\end{enumerate}

\end{theorem}

\section*{Acknowledgment}
We would like to express our deep appreciation to G. Ancona, J. Ayoub, B. Kahn, and L. Migliorini for useful discussions and explanations. We are thankful to C. Voisin for comments and explanations related to the degeneration method.\\
The first named author is grateful to Chia-Fu Yu for his kind support.\\
The second named author would like to thank L. Barbieri Viale for steady encouragements.\\
This paper has emerged out of the notes of a lecture given by the second named author during winter semester 2017-18 at math department of Osnabr\"uck university, and for that she wants to thank the organizers. Especially she gratefully acknowledges O. R\"ondig for his kind support.

\tableofcontents

\section{Notations and Conventions}\label{Notation and Conventions}
Throughout this article we denote by
\begin{tabbing}

$\genericG_\nu:=\FG\times_C\Spec Q_\nu,$\; \=\kill

$k$\> is a perfect field,\\[1mm]

$\textbf{Sch}_k$ (resp. $\textbf{Sm}_k$)\> the category of schemes (resp. smooth schemes) of finite type over $k$,\\[1mm]

$\BF_q$\> a finite field with $q$ elements of characteristic $p$,\\[1mm]
$C$\> a smooth projective geometrically irreducible curve over $\BF_q$,\\[1mm]
$Q:=\BF_q(C)$\> the function field of $C$,\\[1mm]

$\BaseOfD$\> a finite field containing $\BF_q$,\\[1mm]

$\wh A:=\BF\dbl z\dbr$\>  the ring of formal power series in $z$ with coefficients in $\BF$ ,\\[1mm]
$\wh Q:=\Quot(\wh A)$\> its fraction field,\\[1mm]

$\nu$\> a closed point of $C$, also called a \emph{place} of $C$,\\[1mm]
$\BF_\nu$\> the residue field at the place $\nu$ on $C$,\\[1mm]

$\wh A_\nu$\> the completion of the stalk $\CO_{C,\nu}$ at $\nu$,\\[1mm]
$\wh Q_\nu:=\Quot(\wh A_\nu)$\> its fraction field,\\[1mm]


\Verkuerzung
{
\\[1mm]
}
{}
\Verkuerzung
{
\noindent
}
\noindent
$\genericG_\nu:=\FG\times_C\Spec \wh Q_\nu,$\; \=\kill
$n\in\BN_{>0}$\> a positive integer,\\[1mm]
$\ul \nu:=(\nu_i)_{i=1\ldots n}$\> an $n$-tuple of closed points of $C$,\\[1mm]
$\BA^\ul\nu$\> the ring of integral adeles of $C$ outside $\ul\nu$,\\[1mm]
$\wh A_\ul\nu$\> the completion of the local ring $\CO_{C^n,\ul\nu}$ of $C^n$ at the closed point $\ul\nu=(\nu_i)$,\\[1mm]
$\Nilp_{\wh A_\ul\nu}:=\Nilp_{\Spf \wh A_\ul\nu}$\> \parbox[t]{0.79\textwidth}{the category of schemes over $C^n$ on which the ideal defining the closed point $\ul\nu\in C^n$ is locally nilpotent,}\\[2mm]
$\Nilp_{\BaseOfD\dbl\zeta\dbr}$\lang{$:=\Nilp_{\hat\BD}$}\> \parbox[t]{0.79\textwidth}{the category of $\BD$-schemes $S$ for which the image of $z$ in $\CO_S$ is locally nilpotent. We denote the image of $z$ by $\zeta$ since we need to distinguish it from $z\in\CO_\BD$.}\\[2mm]
$\FG$\> a smooth affine group scheme of finite type over $C$,\\[1mm]
$G$\> generic fiber of $\FG$,\\[1mm]
$\BP_\nu:=\FG\times_C\Spec \wh A_\nu,$ \> the base change of $\FG$ to $\Spec \wh A_\nu$,\\[1mm]
$\genericG_\nu:=\FG\times_C\Spec \wh Q_\nu,$ \> the generic fiber of $\BP_\nu$ over $\Spec Q_\nu$,\\[1mm]

$\BP$\> a smooth affine group scheme of finite type over $\BD=\Spec\BaseOfD\dbl z\dbr$,\\[1mm] 
$\genericG$\> the generic fiber of $\BP$ over $\Spec\BaseOfD\dpl z\dpr$.
\end{tabbing}

\noindent
Let $S$ be an $\BF_q$-scheme and consider an $n$-tuple $\ul s:=(s_i)_i\in C^n(S)$. We denote by $\Gamma_\ul s$ the union $\bigcup_i \Gamma_{s_i}$ of the graphs $\Gamma_{s_i}\subseteq C_S$. \\

\noindent

\noindent
We denote by $\sigma_S \colon  S \to S$ the $\BF_q$-Frobenius endomorphism which acts as the identity on the points of $S$ and as the $q$-power map on the structure sheaf. Likewise we let $\hat{\sigma}_S\colon S\to S$ be the $\BaseOfD$-Frobenius endomorphism of an $\BaseOfD$-scheme $S$. We set
\begin{tabbing}
$\genericG_\nu:=\FG\times_C\Spec Q_\nu,$\; \=\kill
$C_S := C \times_{\Spec\BF_q} S$ ,\> and \\[1mm]
$\sigma := \id_C \times \sigma_S$.
\end{tabbing}

\noindent
Let $H$ be a sheaf of groups (for the \'etale topology) on a scheme $X$. In this article a (\emph{right}) \emph{$H$-torsor} (also called an \emph{$H$-bundle}) on $X$ is a sheaf $\CG$ for the \'etale topology on $X$ together with a (right) action of the sheaf $H$ such that $\CG$ is isomorphic to $H$ on a \'etale covering of $X$. Here $H$ is viewed as an $H$-torsor by right multiplication. \\

\noindent
For $X$ in $\textbf{Sch}_k$, let $Ch_i(X)$ and $Ch^i(X)$ denote Fulton's $i$-th Chow groups and let $Ch_\ast(X):=\oplus_i Ch_i(X)$ (resp. $Ch^\ast(X):=\oplus_i Ch^i(X)$).  \\

For $X$ in $\textbf{Sch}_k$, $H^i(X)$ denotes $H^i(\ol X,\BQ_\ell)$ where $\ol X = X \otimes \ol k$ and $\ell\neq p$ is a prime number. We denote $H_i^{BM}
(X)$ the Borel-Moore homology
theory companion to $H^i(X)$. Note that there
is a natural cycle class map $cl : Ch_\ast(X) \to H_{2\ast}^{BM}(X)$.

\noindent
To denote the motivic categories over $k$, such as 

$$
\textbf{DM}_{gm}(k),~ \textbf{DM}_{gm}^{eff}(k),~ \textbf{DM}_{-}^{eff}(k),~\text{etc.} 
$$  
and the functors $M:\textbf{Sch}_k\rightarrow \textbf{DM}_{gm}^{eff}(k)$ and $M^c:\textbf{Sch}_k\rightarrow \textbf{DM}_{gm}^{eff}(k)$ we use the same notation that was introduced  in \cite{VV}. 
We assume coefficients in $\BQ$. 
\\

For the definition of the geometric motives with compact support in positive characteristic we also refer to \cite[Appendix B]{H-K}.\\

\section{Preliminaries}\label{SectPreliminaries}

Recall that a Shimura data $(\BG,X,K)$ consists of a reductive group $\BG$ over $\BQ$ with center $Z$, $\BG(\BR)$-conjugacy class $X$ of homomorphisms $\BS\to G_\BR$ for the Deligne torus $\BS$ and a compact open sub-group $K\subseteq \BG(\BA_f)$, subject to certain conditions; see \cite{MilneShimura}. 

In definition-remark \ref{Def-RemNblaH} we will recall the  analogues picture over function fields. Let us first recall the notion of \emph{local boundedness condition}, introduced in \cite[Definition~4.8]{AH_Local}. To this purpose we first recall some technical background preliminaries. 

\bigskip

The \emph{group of positive loops associated with $\BP$} is the infinite dimensional affine group scheme $L^+\BP$ over $\BaseOfD$ whose $R$-valued points for an $\BaseOfD$-algebra $R$ are $L^+\BP(R):=\BP(R\dbl z\dbr):=\BP(\BD_R)$.
The \emph{group of loops (resp. positive loops) associated with $\genericG$ (resp. $\BP$)} is the $\fpqc$-sheaf of groups $L\genericG$ (resp.  is the infinite dimensional affine group scheme $L^+\BP$ over $\BaseOfD$) over $\BaseOfD$ whose $R$-valued points for an $\BaseOfD$-algebra $R$ are $L\genericG(R):=\genericG(R\dpl z\dpr):=\genericG(\dot{\BD}_R):=\Hom_{\dot\BD}(\dot\BD_R,\genericG)$ (resp. $L^+\BP(R):=\BP(R\dbl z\dbr):=\BP(\BD_R)$)
where we write $R\dpl z\dpr:=R\dbl z \dbr[\frac{1}{z}]$ and $\dot{\BD}_R:=\Spec R\dpl z\dpr$. It is representable by an ind-scheme of ind-finite type over $\BaseOfD$; see \cite[\S\,1.a]{PR2}.
Let $\scrH^1(\Spec \BaseOfD,L^+\BP)\,:=\,[\Spec \BaseOfD/L^+\BP]$ (respectively $\scrH^1(\Spec \BaseOfD,L\genericG)\,:=\,[\Spec \BaseOfD/L\genericG]$) denote the classifying space of $L^+\BP$-torsors (respectively $L\genericG$-torsors). It is a stack fibered in groupoids over the category of $\BaseOfD$-schemes $S$ whose category $\scrH^1(\Spec \BaseOfD,L^+\BP)(S)$ consists of all $L^+\BP$-torsors (resp.\ $L\genericG$-torsors) on $S$. The inclusion of sheaves $L^+\BP\subset L\genericG$ gives rise to the natural 1-morphism 
\begin{equation}\label{EqLoopTorsor}
\L\colon\scrH^1(\Spec \BaseOfD,L^+\BP)\longto \scrH^1(\Spec \BaseOfD,L\genericG),~\CL_+\mapsto \CL\,.
\end{equation}

The affine flag variety $\SpaceFl_\BP$ is defined to be the ind-scheme representing the $fpqc$-sheaf associated with the presheaf
$$
R\;\longmapsto\; L\genericG(R)/L^+\BP(R)\;=\;\BP\left(R\dpl z \dpr \right)/\BP\left(R\dbl z\dbr\right).
$$ 
on the category of $\BaseOfD$-algebras. Note that $\SpaceFl_\BP$ is ind-quasi-projective over $\BaseOfD$, and hence ind-separated and of ind-finite type over $\BaseOfD$, according to Pappas and Rapoport \cite[Theorem~1.4]{PR2}. Additionally, they show that the quotient morphism $L\genericG \to \SpaceFl_\BP$ admits sections locally for the \'etale topology. They proceed as follows. When $\BP = \SL_{r,\BD}$, the \fpqc-sheaf $\CF\ell_\BP$ is called the \emph{affine Grassmanian}. It is an inductive limit of projective schemes over $\BF$, that is, ind-projective over $\BF$; see \cite[Theorem~4.5.1]{B-D}.  By \cite[Proposition~1.3]{PR2} and \cite[Proposition~2.1]{AH_Global} there is a faithful representation $\BP\to\SL_r$ with quasi-affine quotient. Pappas and Rapoport show in the proof of \cite[Theorem~1.4]{PR2} that $\CF\ell_\BP \to \CF\ell_{\SL_r}$ is a locally closed embedding, and moreover, if $\SL_r /\BP$ is affine, then $\CF\ell_\BP \to \CF\ell_{\SL_r}$ is even a closed embedding and $\CF\ell_\BP$ is ind-projective. Moreover, if the fibers of $\BP$ over $\BD$ are geometrically connected, it was proved by Richarz \cite[Theorem A]{Richarz13} that $\CF\ell_\BP$ is ind-projective if and only if $\BP$ is a parahoric group scheme in the sense of Bruhat and Tits \cite[D\'efinition 5.2.6]{B-T}.\\

\noindent
Fix an algebraic closure $\BaseOfD\dpl\zeta\dpr^\alg$ of $\BaseOfD\dpl\zeta\dpr$. Since its ring of integers is not complete, we prefer to work with finite extensions of discrete valuation rings $R/\BaseOfD\dbl\zeta\dbr$ such that $R\subset\BaseOfD\dpl\zeta\dpr^\alg$. For such a ring $R$ we denote by $\kappa_R$ its residue field, and we let $\Nilp_R$ be the category of $R$-schemes on which $\zeta$ is locally nilpotent. We also set $\wh{\SpaceFl}_{\BP,R}:=\SpaceFl_\BP\whtimes_{\BaseOfD}\Spf R$ and $\wh{\SpaceFl}_\BP:=\wh{\SpaceFl}_{\BP,\BaseOfD\dbl\zeta\dbr}$. Before we can define (local) ``bounds'' let us make the following definition.

\begin{definition}\label{DefBDLocal}
\begin{enumerate}
\item \label{DefBDLocal_A}
For a finite extension of discrete valuation rings $\BaseOfD\dbl\zeta\dbr\subset R\subset\BaseOfD\dpl\zeta\dpr^\alg$ we consider closed ind-subschemes $\wh Z_R\subset\wh{\SpaceFl}_{\BP,R}$. We call two closed ind-subschemes $\wh Z_R\subset\wh{\SpaceFl}_{\BP,R}$ and $\wh Z'_{R'}\subset\wh{\SpaceFl}_{\BP,R'}$ \emph{equivalent} if there is a finite extension of discrete valuation rings $\BaseOfD\dbl\zeta\dbr\subset\wt R\subset\BaseOfD\dpl\zeta\dpr^\alg$ containing $R$ and $R'$ such that $\wh Z_R\whtimes_{\Spf R}\Spf\wt R \,=\,\wh Z'_{R'}\whtimes_{\Spf R'}\Spf\wt R$ as closed ind-subschemes of $\wh{\SpaceFl}_{\BP,\wt R}$.

\item
Let $\wh Z=[\wh Z_R]$ be an equivalence class of closed ind-subschemes $\wh Z_R\subset\wh{\SpaceFl}_{\BP,R}$ and let $G_{\wh Z}:=\{\gamma\in\Aut_{\BaseOfD\dbl\zeta\dbr}(\BaseOfD\dpl\zeta\dpr^\alg)\colon \gamma(\wh Z)=\wh Z\,\}$. We define the \emph{ring of definition $R_{\wh Z}$ of $\wh Z$} as the intersection of the fixed field of $G_{\wh Z}$ in $\BaseOfD\dpl\zeta\dpr^\alg$ with all the finite extensions $R\subset\BaseOfD\dpl\zeta\dpr^\alg$ of $\BaseOfD\dbl\zeta\dbr$ over which a representative $\wh Z_R$ of $\wh Z$ exists.

\item

We define a \emph{local bound} (LB) to be an equivalence class $\wh Z:=[\wh Z_R]$ of closed ind-subschemes $\wh Z_R\subset\wh{\SpaceFl}_{\BP,R}$, such that all the ind-subschemes $\wh Z_R$ are stable under the left $L^+\BP$-action on $\SpaceFl_\BP$, and the special fibers $Z_R:=\wh Z_R\whtimes_{\Spf R}\Spec\kappa_R$ are quasi-compact subschemes of $\SpaceFl_\BP\whtimes_{\BaseOfD}\Spec\kappa_R$. The ring of definition $R_{\wh Z}$ of $\wh Z$ is called the \emph{reflex ring} of $\wh Z$. 

\end{enumerate}
\end{definition}

\begin{remark}\label{RemSpecialFiber}
Note that the Galois descent for closed ind-subschemes of $\SpaceFl_\BP$ is effective, thus the representative $Z_R$ of a bound $\hat{Z}$ arise by base change from a unique closed subscheme $Z\subset\SpaceFl_\BP\whtimes_\BaseOfD\kappa_{R_{\wh Z}}$. We call $Z$ the \emph{special fiber} of the bound $\wh Z$. It is a quasi-projective scheme over $\kappa_{R_{\wh Z}}$, and even projective when $\BP$ is parahoric. 
\end{remark}

\begin{example}\label{DefIwahori-Weyl}
 Assume that the generic fiber $\genericG$ of $\BP$ over $\Spec\BaseOfD\dpl z\dpr$ is connected reductive. Consider the base change $\genericG_L$ of $\genericG$ to $L=\BaseOfD^\alg\dpl z\dpr$. Let $S$ be a maximal split torus in $\genericG_L$ and let $T$ be its centralizer. Since $\BaseOfD^\alg$ is algebraically closed, $\genericG_L$ is quasi-split and so $T$ is a maximal torus in $\genericG_L$. Let $N = N(T)$ be the normalizer of $T$ and let $\CT^0$ be the identity component of the N\'eron model of $T$ over $\CO_L=\BaseOfD^\alg\dbl z\dbr$.

The \emph{Iwahori-Weyl group} associated with $S$ is the quotient group $\wt{W}= N(L)\slash\CT^0(\CO_L)$. It is an extension of the finite Weyl group $W_0 = N(L)/T(L)$ by the coinvariants $X_\ast(T)_I$ under $I=\Gal(L^\sep/L)$:
$$
0 \to X_\ast(T)_I \to \wt W \to W_0 \to 1.
$$
By \cite[Proposition~8]{H-R} there is a bijection
\begin{equation}\label{EqSchubertCell}
L^+\BP(\BaseOfD^\alg)\backslash L\genericG(\BaseOfD^\alg)/L^+\BP(\BaseOfD^\alg) \isoto \wt{W}^\BP  \backslash \wt{W}\slash \wt{W}^\BP
\end{equation}
where $\wt{W}^\BP := (N(L)\cap \BP(\CO_L))\slash \CT^0(\CO_L)$, and where $LP(R)=P(R\dbl z\dbr)$ and $L^+\BP(R)=\BP(R\dbl z\dbr)$ are the loop group, resp.\ the group of positive loops of $\BP$; see \cite[\S\,1.a]{PR2}, or \cite[\S4.5]{B-D} when $\BP$ is constant. Let $\omega\in \wt{W}^\BP\backslash \wt{W}/\wt{W}^\BP$ and let $\BaseOfD_\omega$ be the fixed field in $\BaseOfD^\alg$ of $\{\gamma\in\Gal(\BaseOfD^\alg/\BaseOfD)\colon \gamma(\omega)=\omega\}$. There is a representative $g_\omega\in L\genericG(\BaseOfD_\omega)$ of $\omega$; see \cite[Example~4.12]{AH_Local}. The \emph{Schubert variety} $\CS(\omega)$ associated with $\omega$ is the ind-scheme theoretic closure of the $L^+\BP$-orbit of $g_\omega$ in $\SpaceFl_\BP\whtimes_{\BaseOfD}\BaseOfD_\omega$. It is a reduced projective variety over $\BaseOfD_\omega$. For further details see \cite{PR2} and \cite{Richarz}. Now we may take $\wh Z=\CS(\omega)\whtimes_\BaseOfD\Spf \BaseOfD\dbl\zeta\dbr$ for a \emph{Schubert variety} $\CS(\omega)\subseteq \CF\ell_\BP$ , with $\omega\in \wt{W}$; see \cite{PR2}.

\end{example}

Below we briefly recall the construction and basic properties of the algebraic stack Hecke and the stack of global $\FG$-shtukas.\\
\noindent
Let $\BF_q$ be a finite field with $q$ elements, let $C$ be a smooth projective geometrically irreducible curve over $\BF_q$, and let $\FG$ be a smooth affine group scheme of finite type over C. 
\noindent

\begin{definition-remark}\label{Def-RemH1} We let $\scrH^1(C,\FG)$ denote the category fibered in groupoids over the category of $\BF_q$-schemes, such that the objects over $S$, $\scrH^1(C,\FG)(S)$, are $\FG$-torsors over $C_S$ (also called $\FG$-bundles) and morphisms are isomorphisms of $\FG$-torsors.  One can prove that the stack $\scrH^1(C,\FG)$ is a smooth Artin-stack locally of finite type over $\BF_q$. Furthermore, it admits a covering $\{\scrH_\alpha^1\}_{\alpha}$ by connected open substacks of finite type over $\BF_q$. The proof for parahoric $\FG$ (with semisimple generic fiber) can be found in \cite[Proposition~1]{Heinloth} and for general case we refer to \cite[Theorem~2.5]{AH_Global}. For a proper closed subscheme $D$ of $C$ one defines \emph{$D$-level structure} on a $\FG$-bundle $\CG$ on $C_S$ to be a trivialization $\psi\colon \CG\times_{C_S}{D_S}\isoto \FG\times_C D_S$ along $D_S:=D\times_{\BF_q}S$. Let $\scrH_D^1(C,\mathfrak{G})$ denote the stack classifying $\FG$-bundles with $D$-level structure, that is, $\scrH_D^1(C,\mathfrak{G})$ is the category fibred in groupoids over the category of $\BF_q$-schemes, which assigns to an $\BF_q$-scheme $S$ the category whose objects are
$$
Ob\bigl(\scrH_D^1(C,\FG)(S)\bigr):=\left\lbrace (\CG,\psi)\colon \CG\in \scrH^1(C,\FG)(S),\, \psi\colon \CG\times_{C_S}{D_S}\isoto \FG\times_C D_S \right\rbrace,
$$ 
and whose morphisms are those isomorphisms of $\FG$-bundles that preserve the $D$-level structure.

\end{definition-remark}

\noindent
Let us recall the definition of the (unbounded ind-algebraic) Hecke stacks.
\begin{definition}\label{Def_Hecke}
Let $Hecke_{n,D}(C,\FG)$ be the stack fibered in groupoids over the category of $\BF_q$-schemes, whose $S$ valued points are tuples $\bigl((\CG,\psi),(\CG',\psi'),\ul\charsect,\tauGlob\bigr)$ where
\begin{itemize}
\item[--] $(\CG,\psi)$ and $(\CG',\psi')$ are in $\scrH_D^1(C,\FG)(S)$,
\item[--] $\ul\charsect:=(\charsect_i)_i \in (C\setminus D)^n(S)$ are sections, and
\item[--] $\tau\colon  \CG_{|_{{C_S}\setminus{\Gamma_{\ul\charsect}}}}'\isoto \CG_{|_{{C_S}\setminus{\Gamma_{\ul\charsect}}}}$ is an isomorphism preserving the $D$-level structures, that is, $\psi\circ\tauGlob=\psi'$.
\end{itemize}
If $D=\emptyset$ we will drop it from the notation. Note that forgetting the isomorphism $\tau$ defines a morphism 
\begin{equation}\label{EqFactors}
Hecke_{n,D}(C,\FG)\to \scrH_D^1(C,\FG)\times \scrH_D^1(C,\FG) \times (C\setminus D)^n.
\end{equation}

\end{definition}

\begin{definition-remark}\label{Def-RemHecke}
\begin{enumerate}
\item\label{indstructureonhecke}
A choice of faithful representation  $\rho\colon  \FG \to \SL(\CV)$ with quasi-affine (resp.\ affine) quotient $\SL(\CV)\slash \FG$, induces an ind-algebraic structure $$\dirlim Hecke_n^{(\rho,\ul\omega)}(C,\FG)$$ on the stack $Hecke_n(C,\FG)$, which is relatively representable over $\scrH^1(C,\FG)\times_{\BF_q}C^n$ by an ind-quasi-projective (resp. ind-projective) morphism. Note that the limit is taken over n-tuples of coweights $\ul\omega=(\omega_i)$ of $\SL(\CV)$. For details see \cite[Definition~3.9 and Proposition 3.10]{AH_Global}. Note that comparing to \cite[Definition~3.9 and Proposition 3.10]{AH_Global} there is a minor change in our notation, which is intended to illustrate the dependence to the choice of representation $\rho$.

\item\label{globaffGr}
The \emph{global affine Grassmannian} $GR_n(C,\FG)$ is the stack fibered in groupoids over the category of $\BF_q$-schemes, whose $S$-valued points are tuples $(\CG,\ul\charsect,\epsilon)$, where $\CG$ is a $\FG$-bundle over $C_S$, $\ul\charsect:=(s_i)_i \in C^n(S)$ and $\epsilon\colon \CG|_{C_S\setminus \Gamma_\ul s}\isoto \FG \times_C (C_S\setminus \Gamma_\ul s)$ is a trivialization. When the curve $C$ and the group $\FG$ are obvious from the context we drop them from notation and write  $GR_n:=GR_n(C,\FG)$. Notice that the global affine Grassmannian $GR_n$ is isomorphic to the fiber product $Hecke_n(C,\FG)\times_{\,\scrH^1(C,\FG),\FG}\Spec\BF_q$ under the morphism sending $(\CG,\ul\charsect,\epsilon)$ to $(\FG_S,\CG,\ul\charsect,\epsilon^{-1})$. Hence, after we fix a faithful representation $\rho\colon\FG\into\SL(\CV)$ and coweights $\ul\omega$, the ind-algebraic structure on $Hecke_n(C,\FG)$, induces an ind-quasi-projective ind-scheme structure on $GR_n$ over $C^n$.
\forget{
\item\label{heckeforparahoricG}
When the group $\FG$ is parahoric then the ind-algebraic stack $Hecke_n(C,\FG)$ is ind-projective over $C^n\times_{\BF_q}\scrH^1(C,\FG)$. This follows from  \cite{Richarz13} Theorem 1.18. 
}
\item
The group of (positive) loops $\FL_n \FG$ (resp. $\FL_n^+\FG$) of $\FG$ is an ind-scheme (resp. a scheme) representing the functor whose $R$-valued points consist of tuples $(\ul \charsect, \gamma)$ where $\ul\charsect:=(s_i)_i\in C^n(\Spec R)$ and $\gamma\in \FG(\dot{\BD}(\Gamma_\ul s))$(resp. $\gamma\in \FG(\BD(\Gamma_\ul s))$). The projection $(\ul s,\gamma)\mapsto \ul s$ defines morphism $\FL_n\FG\to C^n$ (resp.~ $\FL_n^+\FG\to C^n$). By the general form of the descent lemma of Beauville-Laszlo \cite[Theorem 2.12.1]{B-L}, the map which sends $(\CG, \ul s, \epsilon)\in GR_n(S)$ to the triple $(\ul s,\wh\CG:=\CG|_{\BD(\Gamma_\ul s)}, \dot{\epsilon}:=\epsilon|_{\dot{\BD}(\Gamma_\ul s)})$ is bijective. Thus the loop groups $\FL_n\FG$ and $\FL_n^+ \FG$ act on $GR_n$ by changing the trivialization on $\dot{\BD}(\Gamma_\ul s)$.

\item\label{DefAltGlobalBC_A} 
We fix an algebraic closure $Q^\alg$ of the function field $Q:=\BF_q(C)$ of the curve $C$. For a finite field extension $Q\subset K$ with $K\subset Q^\alg$ we consider the normalization $\wt C_K$ of $C$ in $K$. It is a smooth projective curve over $\BF_q$ together with a finite morphism $\wt C_K\to C$.
For a finite extension $K$ as above, we consider closed ind-subschemes $Z$ of $GR_n\times_{C^n} \wt C_K^n$.
We call two closed ind-subschemes $Z_1\subseteq GR_n\times_{C^n} \wt C_{K_1}^n$ and $Z_2\subseteq GR_n\times_{C^n} \wt C_{K_2}^n$ \emph{equivalent} if there is a finite field extension $K_1.K_2\subset K'\subset Q^\alg$ with corresponding curve $\wt C_{K'}$ finite over $\wt C_{K_1}$ and $\wt C_{K_2}$, such that $Z_1 \times_{\wt C_{K_1}^n} \wt C_{K'}^n=Z_2 \times_{\wt C_{K_2}^n} \wt C_{K'}^n$ in $GR_n\times_{C^n} C_{K'}^n$. 

\item\label{DefAltGlobalBC_B}
Let $\CZ=[Z_K]$ be an equivalence class of closed ind-subschemes $Z_K \subseteq GR_n\times_{C^n} \wt C_K^n$ and let $G_\CZ:= \{g \in \Aut(Q^\alg/Q): g^\ast(\CZ)=\CZ\}$. We define the \emph{field of definition} $Q_\CZ$ of $\CZ$ as the intersection of the fixed field of $G_\CZ$ in $Q^\alg$ with all the finite extensions over which a representative of $\CZ$ exists.

\item\label{DefAltGlobalBC_C}
We define a \emph{global bound} (GB) to be an equivalence class $\CZ:=[Z_K]$ of closed subscheme $Z_K\subset GR_n\times_{C^n} \wt C_K^n$, such that all the ind-subschemes $Z_K$ are stable under the left $\FL_n^+\FG$-action on $GR_n$. The field of definition $Q_\CZ$ (resp. the curve of definition $C_\CZ:=\wt C_{Q_\CZ}$) of $\CZ$ is called the \emph{reflex field} (resp. \emph{reflex curve}) of $\CZ$.

\item\label{heckeislocallGr}
Consider the stacks $Hecke_n(C,\FG)$ and $GR_n\times \scrH^1(C,\FG)$ as families over $C^n \times \scrH^1(C,\FG)$, via the projections $(\CG,\CG',\ul s,\tauGlob)\mapsto (\ul s,\CG')$ and $(\wt{\CG},\ul s,\wt\tauGlob)\times \CG'\mapsto (\ul s,\CG')$ respectively. They are locally isomorphic with respect to the \'etale topology on $C^n\times \scrH^1(C,\FG)$. See \cite[Proposition~2.0.11]{AH_LM}, which is a slight generalization of \cite[Lemma 4.11]{Var}.

\end{enumerate}
\end{definition-remark}

Assume that we have two morphisms $f,g\colon X\to Y$ of schemes or stacks. We denote by $\equi(f,g\colon X\rightrightarrows Y)$ the pull back of the diagonal under the morphism $(f,g)\colon X\to Y\times_\BZ Y$, that is $\equi(f,g\colon X\rightrightarrows Y)\,:=\,X\times_{(f,g),Y\times Y,\Delta}Y$ where $\Delta=\Delta_{Y/\BZ}\colon Y\to Y\times_\BZ Y$ is the diagonal morphism.

\noindent
Below we recall the construction and basic properties of the stack of global $\FG$-shtukas.

\begin{definition-remark}\label{Def-RemNblaH}
\item[a)]
We define the \emph{moduli stack} $\nabla_n\scrH_D^1(C,\mathfrak{G})$ \emph{of global $\FG$-shtukas with $D$-level structure} to be the preimage in $Hecke_{n,D}(C,\FG)$ of the graph of the Frobenius morphism on $\scrH^1(C,\FG)$. In other words
$$
 \nabla_n\scrH_D^1(C,\mathfrak{G}):=\equi(\sigma_{\scrH_D^1(C,\mathfrak{G})} \circ pr_1,pr_2\colon  Hecke_{n,D}(C,\FG)\rightrightarrows \scrH_D^1(C,\FG)),
$$ 
where $pr_i$ are the projections to the first, resp.\ second factor in \eqref{EqFactors}. Each object $\ul{\cG}$ of $\nabla_n\scrH_D^1(C,\FG)(S)$ is called a \emph{global $\FG$-shtuka with $D$-level structure over $S$} and the corresponding sections $\ul\charsect:=(\charsect_i)_i$ are called the \textit{characteristic sections} (or simply \textit{characteristics}) of $\ul \CG$.\\ 
\noindent
\item[b)]
(ind-algeraic structure on $\nabla_n \scrH_D^1(C, \FG)$)
Let $D$ be a proper closed subscheme of $C$. The ind-algebraic structure on $Hecke$ induces an ind-algebraic structure $\nabla_n \scrH_D^1(C, \FG)=\dirlim\nabla_n^{(\rho, \ul \omega)} \scrH_D^1(C, \FG)$ over $(C\setminus D)^n$ which is ind-separated and locally of ind-finite type. The stacks $\nabla_n^{(\rho, \ul \omega)} \scrH_D^1(C, \FG)$ are Deligne-Mumford. Moreover, the forgetful morphism $$\nabla_n \scrH_D^1(C, \FG)\rightarrow \nabla_n \scrH^1(C, \FG)\times_{C^n}(C\setminus D)^n$$ is surjective and a torsor under the finite group $\FG(D)$. See \cite[Theorem 3.15]{AH_Global}.

\end{definition-remark}

\section{Discussion On The Motives Of The Stacks Of Shtukas and Their Local Models}\label{SectMotiveOfModuliOfShtukas}

In this section we first study the motive of the Rapoport-Zink type local model for the stack of $\FG$-shtukas. For this we first need to explain some necessary motivic background. We then discuss the motivic categories over the moduli stacks of $\FG$-shtukas and we further continue by discussing their intersection motives. Finally we study the motive of the local models in the global setup.

\subsection{Motives Of R-Z Local Models and Semi-simplicity Of Frobenius}\label{SubSectSSFrobLM}

Recall that Grothendieck in 1960's revealed his significant observation by proposing a unifying method to describe the essence of different cohomology theories. In his article \cite{Gro}, he further proposed several conjectures, called standard conjectures, which in particular imply that his construction of the category of (pure) motives gives a semi-simple abelian category. He further observed that these conjectures imply the Riemann hypothesis part of the Weil conjectures. 

\noindent
Let $X$ be a smooth projective variety over a finite field $\BF_q$ of characteristic $p$ and let $\ell$ be a prime number different from $p$. Believing the standard conjectures, one can see that the action of the Frobenius on the \'etale cohomology $H_{\text{\'et}}^i(X_{\ol{\BF}_q},\BQ_\ell)$ is semisimple. For the case of abelian varieties this follows from the Weil's work on the Riemann hypothesis. Namely, fixing a  polarization of $\CA$ induces a Rosati involution $\phi\mapsto\phi^\dagger$ on the endomorphism ring $\BQ\mathrm{End}(\CA):= \mathrm{End} (\CA)\otimes \mathbb{Q}$, which further induces a bilinear $\phi\mapsto \Tr(\phi\phi^\dagger)$ form on sub-$\BQ$-algebra $F:=\BQ[\pi]\subseteq\BQ\mathrm{End}(\CA)$ generated by the gemotric Frobenius $\pi$, where the latter is positive definite according to the Riemann hypothesis. For general case, to formulate this in terms of the positivity of a bilinear form at the level of cohomology, one needs standard conjecture of Lefschetz and Hodge type. As another evidence for semi-simplicity of Frobenius, one could mention the case of $K3$ surfaces. This follows from results of Deligne; see \cite{DeligneWeilConjK3}. \\


There are also other sources of evidences, for which, the semi-simplicity conjecture can be justified in more elementary way. For example recall that according to $\BA^1$-homotopy theory, one may collapse affine sub-spaces of a variety. This naturally suggests to consider those varieties that can be paved by affine spaces and further to reduce more subtle cases to this case using tools such as decomposition theorem (in the sense of \cite{BBD} and also in the motivic sense of \cite{CoHa} and \cite{C-M}), Leray-Hirsch theorem and etc. \forget{For this part we use the machinery of motives and  weight (or slice) filtration. Our approach relies on \cite{H-K} and the authors previous work \cite{AH17}.\\  

In this case one can precisely describe the action on the on the cohomology of $X$. In particular this is the case when $X$ is cellular. } Recall that

\begin{definition}
i)
We say that $X \in \textbf{Sch}_k$ is \emph{relatively cellular} if it admits a filtration by its closed subschemes: 

\begin{equation}\label{EqRelCell}
\emptyset = X_{-1} \subset X_0 \subset \dots \subset X_n = X  
\end{equation}

such that $U_i := X_i \setminus X_{i-1} \to Y_i$ is an affine bundle of relative dimension $d_i$.\\
ii) We say that $X$ is cellular if $Y_i=\Spec k$ for all $i$. 

\end{definition}

\noindent
Let us recall that the motive of a relative cellular variety filtered as above, admit a decomposition as sum of motives $M(Y_i)$ with relevant shift and twist. 

\begin{proposition}\label{Prop_DecompRelCel}
i) 
Assume that $X\in \textbf{Sch}_k$ is relatively cellular with a filtration as in (\ref{EqRelCell}). Then we have the following decomposition
$$
M^c(X)=\bigoplus_i M^c(Y_i)(d_i)[2d_i].
$$

In particular when $X$ is cellular then we have

$$
M^c(X)=\bigoplus_i\BZ(d_i)[2d_i].\\
$$

ii) Similarly when $X$ is cellular then $\pi_!(\BQ_\ell)$ is direct sum of the $IC$-sheaves of the form $\BQ_\ell(m)[n]$. Here $\pi$ denotes the structure morphism to $k$. 
\end{proposition}

\begin{proof}
i) We prove by induction on $\dim X$. Consider the following distinguished triangle
$$
M^c(X_{j-1}) \rightarrow M^c(X_j) \xrightarrow{\;g_j} M^c(U_j) \rightarrow M^c(X_{j-1})[1].
$$
The closure of the graph of $p_j : U_j \rightarrow Y_j$ in $X_j \times Y_j$. This defines a cycle in $CH_{dim~X_j}(X_j\times Y_j)$ and since $Y_j$ is smooth this induces the following morphism 
$$
\gamma_j : M^c(Y_j)(d_j)[2d_j] \rightarrow M^c(X_j),
$$
by \cite[Chap. 5, Theorem 4.2.2.3) and Proposition 4.2.3]{VV}, such that $g_j \circ \gamma_j =p_j^\ast$. Thus the above distinguished triangle splits and we conclude by induction hypothesis.\\

ii)  By induction on dim $X$, we may assume that the statement holds for $X\setminus U$, for an open affine space $U$. The statement follows from splitting of the canonical distinguished triangle $\pi_!(\BQ_\ell|_U) \to \pi_!(\BQ_\ell) \to \pi!(\BQ_\ell|_{X\setminus U})$.

\end{proof}

\begin{remark}
It can be seen that projective homogeneous spaces admit cellular decomposition, e.g. see \cite{Koeck}. Another well-known class of examples, according to Bialynicki-Birula decomposition, comes from projective varieties with $\BG_m$-action.  Also toric varieties and wonderful compactification $\ol G$ of a split reductive group $G$, provides another important sub-specious of cellular varieties; see for example \cite[proposition 3.5]{AH17}.
\end{remark}

From the above proposition \ref{Prop_DecompRelCel} it is clear that the motive corresponding to a cellular schemes lies in the category of \emph{pure Tate motives}. Recall that

\begin{definition} 

An object of $\textbf{DM}_{gm}(k)$ is called pure Tate motive if it is a (finite) direct sum of copies of $\BZ(n)[2n]$ for $n\in \BZ$. 
\end{definition}

\noindent
To proceed further and study the geometry of more complicated schemes we first recall the slice filtration and motivic fundamental invariant according to \cite{H-K}.

Recall that a t-structure on a triangulated category $\cC$ is the data of two full subcategories $\cC^{\geq 0}$ and $\cC^\leq0$ satisfying the following properties: 
\begin{enumerate}
\item
(T1) If $\cC^{\geq n}=\cC^{\geq 0}[-n]$ and $\cC^{\leq n}=\cC^{\leq 0}[-n]$, then $\cC^{\geq 1}\subseteq \cC^{\geq 0}$ and $\cC^{\leq -1} \subseteq \cC^{\leq 0}$, 
\item
(T2) For any $\CF \in \cC^{\leq 0},G \in \cC^{\geq 1}$, $\Hom_\cC(\CF,\CG) = 0$,
\item
(T3) For any $\CF \in C$, there exists a distinguished triangle
\begin{equation}
\xymatrix
   {  
      \CF_0 \ar[rr] & & \CF\ar[dl]\\
     & \CF_1\ar[ul]^{[1]}  &   
   }
\end{equation}
 
such that $\CF_0 \in \cC^{\leq 0}$ and $\CF_1 \in \cC^{\geq 1}$. 
\end{enumerate}
\noindent
The heart of the t-structure is the full subcategory $\cC^{\leq 0} \cap \cC^{\geq0}$.

\noindent
Consider the inclusions of the full subcategories $i : \cC^{\leq n} \to \cC$(resp. $i': \cC^{\geq n} \to \cC$). Then, there exist a truncation functors $\tau^{\leq n} : \cC \to \cC^{\leq n}$ (resp. $\tau^{\geq n} : \cC \to \cC^{\geq n}$) and such that for any $Y \in \cC^{\leq n}$  (resp. $Y \in \cC^{\geq n}$) and any $X \in \cC$, we have isomorphisms $\Hom_{\cC\leq n}(Y,\tau^{\leq n} X) \to \Hom_\cC(i(Y),X)$ (resp. $\Hom_{\cC^{\geq n}}(\tau^{\geq n}X,Y ) \to \Hom_\cC(X,i'(Y ))$).

The category $\textbf{DM}_{-}^{eff}(k)$ has a (non degenerate) t-structure inherited from the derived category  $D^-(Shv_{Nis}(SmCor(k)))$ of Nisnevich sheaves whose heart
is the abelian category $HI(k)$ of homotopy invariant Nisnevich sheaves with
transfers on $Sm/k$. This t-structure is not the desired (conjectural) “motivic
t-structure” on $\textbf{DM}_{-}^{eff}(k)$ whose heart is the abelian category of (effective) mixed motives over $k$. For further details in this direction see \cite[Chapter~5]{VV}.

\noindent
Let us now recall the following alternative construction from \cite{H-K} of so-called slice filtration.

For $M$ in $\textbf{DM}_{-}^{eff}(k)$, one defines the following triangulated functor
$$
\nu^{\geqslant n} M :=\ul{\Hom}(\BZ (n),M)(n).
$$

\noindent
By adjunction there are morphisms $a^n: \nu^{\geqslant} M\to M$ and $f^n:\nu^{\geqslant n} M\to \nu^{\geqslant n-1} M$.
Where $a^n$ corresponds to identity via 
$$
\CD
\Hom(\ul\Hom(\BZ(n),M),\ul\Hom(\BZ(n),M))\\~~~~~~~~~~~~~~=\Hom(\ul\Hom(\BZ(n),M)(n),M)\\=\Hom(\nu^{\geqslant n} M,M)
\endCD
$$

and

$$
\CD
\Hom(\ul\Hom(\BZ(n),M),\ul\Hom(\BZ(n),M))\\=
\Hom(\ul\Hom(\BZ(n),M)\otimes\BZ(1),\Hom(\BZ(n-1),M))\\
=\Hom(\ul\Hom(\BZ(n),M)\otimes\BZ(n),\ul\Hom(\BZ(n-1),M)(n-1))\\=\Hom(\nu^{\geqslant n} M,\nu^{\geqslant n-1} M).
\endCD
$$
Accordingly

\begin{definition}\label{DefSliceFilt}
 Define $\nu_{< n}M = \nu_{\leq n-1}M$ and $\nu_nM$ as the objects given by the following triangles

\begin{equation}\label{Huber_Kahn_Filt}
\xymatrix
   {  
      \nu^{\geqslant n} M \ar[rr]^{a^n} & & M\ar[dl]\\
     & \nu_{< n} M\ar[ul]^{[1]}  &   
   }
\end{equation}

 The natural transformations $a_n : Id\to \nu_{<n}$ factor canonically through natural transformations $f_n : \nu_{<n+1} \to \nu_{<n}$, which allows to define

\begin{equation}\label{Huber_Kahn_Filt}
\xymatrix
   {  
      \nu_{<n+1} M \ar[rr]^{f_n} & & \nu_{<n}M\ar[dl]\\
     & \nu_{n}M\ar[ul]^{[1]}  &   
   }
\end{equation}

\noindent
Now define the following functors

 $$
c_n: \textbf{DM}_{gm}^{eff}(k) \to \textbf{DM}_{gm}^{eff}(k),
 $$

\noindent
where $\nu_n M:= c_n(M)(n)[2n]$. The object $c_n(M)$ is called the \emph{$n$'th motivic fundamental invariant} of $M$.  

\end{definition}


\begin{definition-remark}\label{DefRemCn}
\begin{enumerate}
\item
Consider the thick tensor subcategory of $\textbf{DM}_{gm}^{eff}(k)$ (resp. $\textbf{DM}_{-}^{eff}(k)$) generated by $\BZ(0)$. It is isomorphic to the full subcategory $\textbf{D}_f^b(\textbf{Ab})$ of the bounded derived category $\textbf{D}^b(\textbf{Ab})$ of abelian groups, consisting of objects with finitely generated cohomology groups; see \cite[Proposition 4.5.]{H-K}.

\item
Define the category of mixed-Tate motives $MT\textbf{DM}_{gm}^{eff}(k)$ as the thick tensor subcategory of $\textbf{DM}_{gm}^{eff}(k)$ generated by $\BZ(0)$ and the Tate object $\BZ(1)$. The above obvious embedding precisely factors through $i: \textbf{D}_f^b(\textbf{Ab})\to MT\textbf{DM}_{gm}^{eff}(k)$ (resp. $i: \textbf{D}^b(\textbf{Ab})\to MT\textbf{DM}_{-}^{eff}(k)$).

\item
The functors $\nu^{\leq n}$ and $\nu^{\geq n}$ restrict to functors $MT\textbf{DM}_{gm}^{eff}(k)\to MT\textbf{DM}_{gm}^{eff}(k)$. Note that in this situation the slice filtration coincides the usual weight filtration.   

\item
 Let $M \in \textbf{DM}_{gm}^{eff}(k)$ and $N \in MT\textbf{DM}_{gm}^{eff}(k)$. There is a K\"unneth isomorphism
$$
\bigoplus_{p +q=n}
c_p(M)\otimes c_q(N)\tilde{\to} c_n(M\otimes N).
$$

\end{enumerate}

\end{definition-remark}

\begin{proposition}\label{PropDecompPureTate}
Let $X$ be a smooth variety such that $M(X)$ is pure Tate. Then there is a natural isomorphism 
$$
M(X)\cong\bigoplus_p c_p(X)(p)[2p].
$$
\end{proposition}

\begin{proof}
See \cite[Proposition~4.6 and Proposition~4.10]{H-K}.
\end{proof}

Consider the functor $c_n(-)$, in the remaining part we deal with the question that when the n'th fundamental motivic invariant $c_n(X):=c_n(M(X))$ lies in the category $\textbf{D}^b(\textbf{Ab})$.

\begin{example}\label{ExDecompPureTate}

\begin{enumerate}
\item

Assume that $M:=M(X)$ is pure Tate. Then we have $c_p(X)=Ch_p(X)[0]=\Hom(M,\BZ(p)[2p])$. 
\item
Let $\Gamma$ be a fiber bundle over $X$ with fiber $F$ as in the theorem \ref{Leray Hirsch for Voevodsky motives} below. Then since $M(F)$ is pure Tate we have

$$
c_n(M(\Gamma))=\bigoplus_{p+q=n}c_p(F)\otimes c_q(M(X)),
$$
with $c_p(F)=Ch_p(F)[0]$. In particular $c_n(M(\Gamma))$ lies in $\textbf{D}^b(\textbf{Ab})$ provided that $c_i(X)\in \textbf{D}^b(\textbf{Ab})$ for $0\leq i\leq n$.
\item
Consider the moduli stack $\CX:=\scrH^1(C,\GL_n)$. Note that this is isomorphic to the stack of vector bundles of rank $n$ on a curve $C$. Let $\CX^d$ denote the substack of $\scrH^1(C, \GL_n)$ parameterizing vector bundles of degree $d$. For $C:=\BP^1$ one sees that $c_n(\CX)$ lies in $\textbf{D}^b(\textbf{Ab})$. This follows from the following formula for the motive of $\CX^d$

$$
\CD
M(\CX^d)=M(B\BG_m) \otimes M(Jac(C)) \otimes_{i=1\dots n-1} Z(C,\BQ(i)[2i])),
\endCD
$$

See \cite{Huskins}. Here $Z(C,\BQ(i)[2i]))$ denotes the motivic zeta function of $C$. Regarding \cite[conjecture 3.4]{Beh07} one expects that a similar fact holds for the motive of $\scrH^1(C,G)$ for split reductive group $G$.
 
\end{enumerate}
\end{example}

\begin{proposition}\label{PropChowMixedTateIsFinite}
Let $X\in \textbf{Sch}_\BF$ and assume that $M^c(X)$ is mixed Tate, then $Ch_\ast(X)$ is finite over $\BZ$.
\end{proposition}

\begin{proof} First assume that the motive $M:=M^c(X)$ lies in the category $\CE_0$ of extensions of $\oplus\BZ(n)[j]$ by $\oplus\BZ(n')[j']$. Then the statement follows by applying $\Hom(\BZ(i)[2i],-)$ to the triangle 
$$
\oplus\BZ(n)[j]\to M\to \oplus\BZ(n')[j'].
$$
Note that since $\BF$ is a finite field we observe that 
$$
\Hom(\BZ(i)[2i],\BZ(n)[j])
$$ 
is finite; e.g. see \forget{\cite{Qui} and \cite{Lev}} \cite[Theorem~40]{Kahn05}.  Now consider the category $\CE_1$ of all $M_1$ that are extensions of $M_0$ by $M_0'$ in $\CE_0$. We may argue as above that $\Hom(\BZ(i)[2i],M_1)$ is finitely generated. One can repeat this process to produce the category $\CE_n$, and observe that $Ch_\ast(X)$ is finitely generated when $M$ lies in $\CE_n$. So it remains to see that $\dirlim \CE_n$ coincides the category of mixed Tate motives, which is clear, because by construction it is a thick subcategory of the category of mixed Tate motives that contains motives of the form $\BZ(i)[j]$. 
\end{proof}

\begin{definition}\label{DefGeoMixedTate}
 An object $M \in DM_{gm}^{eff}(F)$ is called \emph{geometrically mixed Tate} if its
restriction to $DM_{gm}^{eff}(F^{sep})$ lies in $MTDM_{gm}^{eff}(F^\sep)$.
\end{definition}

\begin{lemma}\label{LemGeoMixedTate}
An object $M \in DM_{gm}^{eff}(F)$ is geometrically mixed Tate if and only if for some finite separable extension $E$ of $F$ the restriction of $M$ to $DM_{gm}^{eff}(E)$ lies in $MTDM_{gm}^{eff}(E)$.
\end{lemma}

\begin{proof}
See \cite[Proposition~5.3]{H-K}.
\end{proof}

Recall from \cite{AH17} that in the motivic context one may formulate the Leray-Hirsch theorem after imposing certain conditions to the fiber.

\begin{theorem}\label{Leray Hirsch for Voevodsky motives}
Let $X$ be a smooth irreducible variety over $k$. Let $\pi:\Gamma \rightarrow X$ be a proper smooth locally trivial fibration with fibre $F$.  Furthermore, assume that $F$ is cellular and satisfies Poincar\'e duality. Then one has an isomorphism
$$
M(\Gamma)\cong \bigoplus_{p \geq 0}CH_p(F) \otimes M(X)(p)[2p]
$$  
in $\textbf{DM}_{gm}^{eff}(k)$.
\end{theorem}

\begin{remark}\label{Rem_IteratedTower}
One can proceed by considering the following situation. Namely, start with a variety $\wt X_0$ and then consider an iterated tower of cellular fibrations

$$
\CD
\tilde{X}_n:=\tilde{X}\\
@V{p_{n-1}}VV\\
\tilde{X}_{n-1}\\
\vdots\\
@VVV\\
\tilde{X}_0
\endCD
$$
over it.  In this case one can recursively compute the cohomology of $\tilde{X}$ by applying Leray-Hirsch theorem. In particular we observe that if $M(\tilde{X}_0)$ is mixed Tate then $M(\tilde{X})$ is mixed Tate as well.

\end{remark}

Following the above constructive method, to treat more complicated singular varieties, one can proceed in the following way. Namely, consider a variety $\CS$ which admit a resolution of singularities $\Sigma\to \CS$ by an iterated tower of cellular fibrations. We then apply the decomposition theorem \cite{BBD} to such resolution. Note that in order to establish this theorem in the motivic context one needs to assume Grothendieck's standard conjectures and Murre's conjecture. Regarding this to implement the decomposition theorem in the motivic context we need to restrict our attention to the cases where the cycle class map is easier to be described.


\begin{proposition}\label{PropMotivicDecomposition}
Assume that $\CS$ admits a resolution of singularities $\Sigma\to\CS$. Consider the cycle class map
$$
cl: A_\ast:=Ch_\ast(\Sigma\times_\CS \Sigma)\to H_{2\ast}^{BM}(\Sigma\times_\CS \Sigma). 
$$
We have the following statements

\begin{enumerate}
\item
Assume $\dim_\BQ A_\ast$ is finite. Note that this is for example the case when $M(\Sigma\times_\CS\Sigma)$ is mixed Tate; see Lemma \ref{PropChowMixedTateIsFinite}. Assume further that $cl$ is surjective and the $\ker cl$ lies inside the Jacobson radical $J$ of $A_\ast$. Then any decomposition of $Rf_\ast\BQ_\Sigma[n]$ lifts to a decomposition of the Chow motive $(X,\Delta_X)$. 

\item
If cycle class map $cl$ is an isomorphism then any decomposition of $Rf_\ast\BQ_\Sigma[n]$ corresponds to a unique decomposition of $(\Sigma,\Delta_\Sigma)$.

\end{enumerate}

\end{proposition}

\begin{proof}
a) As $A_\ast$ is finite, some power of $J$ vanishes. This allows to lift an idempotent $A_\ast$ modulo $J$. Furthermore, there is an isomorphism 

\begin{equation}\label{eqcyCoHaIsom}
\End_{D_{cc}^b(\CS)}(Rf_\ast\BQ_\Sigma[n])\to H_{2n}^{BM}(\Sigma\times_\CS\Sigma),
\end{equation}

see \cite[Lemma 2.21]{CoHa}. Therefore an idempotent of $\End_{D_{cc}^b(\CS)}(Rf_\ast\BQ_\Sigma[n])$ induces an idempotent of $H_{2n}^{BM}(\Sigma\times_\CS\Sigma)$. This lifts to an idempotent in $A_\ast$.

b) This follows from the isomorphism \ref{eqcyCoHaIsom} followed by inverse of cycle class map.

\end{proof}

\begin{definition}\label{DefCatNett}
Let $\CC_{\leq n}$ be the full subcategory of the category of $\textbf{Sch}_k$ whose objects satisfy the following conditions. Namely, every $\CS\in\CC_{\leq n}$ admits a surjective morphism $\pi:\Sigma:=\Sigma(\CS)\to \CS$ in $\textbf{Sch}_k$ such that
\begin{itemize} 
\item[a)]
$\Sigma$ is smooth and the $m$-th motivic invariant $c_m(M(\Sigma))$ lies in $\textbf{D}_f^b(\textbf{Ab})$ for all $m\leq n$ and,\\
\item[b)]
$\pi: \Sigma\to \CS$ admits a stratification $\CS=\bigcup_\alpha \CU_\alpha$ such that:
\begin{itemize}
\item[-]
$\pi: \pi^{-1}(\CU_\alpha)\to \CU_\alpha$ is a locally trivial fiber bundle with cellular fiber $F_\alpha$,
\item[-]
 $\ol \CU_\alpha$ lies in $\CC_{\leq n}$ for $\alpha\neq \alpha_0$, where $\CU_{\alpha_0}$ is the open stratum.
\end{itemize}
\end{itemize}

One can strengthen the above conditions and define the category $\CC_\ll$ by requiring that the motivic invariants of $\Sigma$ lie in $\textbf{D}_f^b(\textbf{Ab})$ for all $n$, and additionally, they vanish when $n$ is sufficiently large.

\end{definition}

\forget{
----------------------------------------
\begin{definition}\label{DefCatNett}
Let $\CC_n$ be the full subcategory of the category of $\textbf{Sch}_k$ whose objects have the following property. Namely, every $\CS\in\CC_n$ admits a surjective morphism $\pi:\Sigma:=\Sigma(\CS)\to \CS$ in $\textbf{Sch}_k$ such that
\begin{itemize} 
\item[a)]
$\Sigma$ is smooth and the n-th motivic invariant $c_n(M(\Sigma))$ lies in $\textbf{D}_f^b(\textbf{Ab})$,\\
\item[b)]
$\pi: \Sigma\to \CS$ admits a stratification $\CS=\bigcup_\alpha \CU_\alpha$ such that:
\begin{itemize}
\item[-]
$\pi: \pi^{-1}(\CU_\alpha)\to \CU_\alpha$ is a locally trivial fiber bundle with cellular fiber $F_\alpha$,
\item[-]
 $\ol \CU_\alpha$ lies in $\CC_n$ for $\alpha\neq \alpha_0$, where $\CU_{\alpha_0}$ is the open stratum.
\end{itemize}
\end{itemize}
 
One can strengthen condition $a)$ by requiring that the motivic invariants of $\Sigma$ lie in $\textbf{D}_f^b(\textbf{Ab})$ for all $n$. We denote the resulting category by $\CC_-$.  Moreover, we denote by $\CC_\ll$ the category obtained by additionally requiring that these motivic invariants vanish for sufficiently large $n$.  
\end{definition}

-------------------------------
}

\begin{remark}\label{RemTowerOfFib}
Due to the geometry of convolution morphisms, it is sometimes useful to consider the following more general situation. Namely, we may allow fibers $F_\alpha$ to be an iterated tower of cellular fiberations as we mentioned earlier in remark \ref{Rem_IteratedTower}.
\end{remark}

\begin{definition} \label{DefMixedArrangement}

Let $X$ be a projective variety over $k$. We say that $X$ is a $c_n^{-1}(\textbf{D}_f^b(\textbf{Ab}))$-configuration if $X=\cup_i X_i$ with irreducible $X_i$'s, such that
\begin{enumerate}
\item[i)] $c_n(M(X_i))$ lies in $\textbf{D}_f^b(\textbf{Ab})$, and
\item[ii)] the union of the elements of any arbitrary subset of $\{X_{ij} := X_i \cap X_j\}_{i \neq j}$ is either a $c_n^{-1}(\textbf{D}_f^b(\textbf{Ab}))$-configuration or empty.

\end{enumerate} 

Furthermore we say that $X$ is a $c^{-1}(\textbf{D}_f^b(\textbf{Ab}))$-configuration if the above holds for every $n>0$. We sometimes strengthen this even further by requiring that $c_n(M(X_i))$  vanishes for large enough $n$. Then $X=\cup_i X_i$ is called a \emph{strict} $c^{-1}(\textbf{D}_f^b(\textbf{Ab}))$-configuration. 

\end{definition}

\begin{example}\label{ExampleBigDiagonal}
Assume that $c_n(Y)$ lie in $\textbf{D}_f^b(\textbf{Ab})$, then the irreducible components of big diagonal $\Delta\subset Y^n$ form a $c_n^{-1}(\textbf{D}_f^b(\textbf{Ab}))$-configuration.
\end{example}

\begin{definition}\label{DefNett}
We say that a bound $\wh Z$ is nett (resp. $c_{\leq n}$-nett) if its special fiber (in the sense of Remark \ref{RemSpecialFiber}) is of finite type, proper and lies in $\CC_{\ll}$ (resp. $\CC_{\leq n}$).
\end{definition}

\forget{
-----------------------------------------
\begin{definition}\label{DefNett}
We say that a bound $\wh Z$ is nett (resp. $c_n$-nett) if its special fiber (in the sense of Remark \ref{RemSpecialFiber}) is of finite type, proper and lies in $\CC_{\ll}$ (resp. $\CC_m$ for all $m\leq n$).
\end{definition}
---------------------------------------
}

\begin{remark}
The above definitions are inspired by the well-known split reductive case. We don't know up to what extent the above definitions may remain useful beyond this case. Although we expect this at least up to some modifications.  
\end{remark}

\begin{theorem}\label{ThmFrobistSemiSimpleOnZ}
Let $\hat{Z}$ be a boundedness condition in the sense of Definition \ref{DefBDLocal}. We have the following statements

a) Assume that $\hat{Z}$ is $c_{\leq n}$-nett boundedness condition in the sense of definition \ref{DefNett}. Then $c_n(M(Z))$ lies in $\textbf{D}_f^b(\textbf{Ab})$.\\

b) Assume that  $\ol Z:=Z\times_k \ol k$ is irreducible and admits a stratified $\ol Z=\coprod_{\beta\in \CB} \ol Z_\beta$ semi-small resolution of singularities $\ol\Sigma\to \ol Z$, with $c_n(M(\ol \Sigma))\in \textbf{D}_f^b(\textbf{Ab})$.  Then $c_n(M(\ol Z))\in \textbf{D}_f^b(\textbf{Ab})\subseteq \textbf{DM}_{gm}^{eff}(\ol k)$.\\

\noindent
Moreover:\\

c) Assume that $\hat{Z}$ is nett (resp. there is  a semi-small resolution $\ol\Sigma\to \ol Z$ with $M(\ol\Sigma)\in MT\textbf{DM}_{gm}^{eff}(\ol k)$). Then $M(Z)$ (resp. $M(\ol Z)$) lies in $MT\textbf{DM}_{gm}^{eff}(k)$ (resp. $MT\textbf{DM}_{gm}^{eff}(\ol k)$). \\

d) In either of the cases mentioned in c), Frobenius acts semi-simply on the cohomology $\Koh^i(\hat{Z},\BQ_\ell)$, here $\ell\neq p$ is a prime number.

\end{theorem}

\forget{
------------------------------------

\begin{theorem}\label{ThmFrobistSemiSimpleOnZ}
Let $\hat{Z}$ be a boundedness condition in the sense of Definition \ref{DefBDLocal}. We have the following statements

a) Assume that $\hat{Z}$ is $c_n$-nett boundedness condition in the sense of definition \ref{DefNett}. Then $c_n(M(Z))$ lies in $\textbf{D}_f^b(\textbf{Ab})$.\\

b) Assume that  $\ol Z:=Z\times_k \ol k$ is irreducible and admits a stratified $\ol Z=\coprod_{\beta\in \CB} \ol Z_\beta$ semi-small resolution of singularities $\ol\Sigma\to \ol Z$, with $c_n(M(\ol \Sigma))\in \textbf{D}_f^b(\textbf{Ab})$.  Then $c_n(M(\ol Z))\in \textbf{D}_f^b(\textbf{Ab})\subseteq \textbf{DM}_{gm}^{eff}(\ol k)$.\\

\noindent
Moreover:\\

c) Assume that $\hat{Z}$ is nett (resp. there is  a semi-small resolution $\ol\Sigma\to \ol Z$ with $M(\ol\Sigma)\in MT\textbf{DM}_{gm}^{eff}(\ol k)$). Then $M(Z)$ (resp. $M(\ol Z)$) lies in $MT\textbf{DM}_{gm}^{eff}(k)$ (resp. $MT\textbf{DM}_{gm}^{eff}(\ol k)$). \\

d) In either of the cases mentioned in c), Frobenius acts semi-simply on the cohomology $\Koh^i(\hat{Z},\BQ_\ell)$, here $\ell\neq p$ is a prime number.

\end{theorem}

--------------------------------
}

\begin{proof}

Let us first prove the following

\begin{lemma}\label{LemMixedArrangement}
Suppose $X=\cup_i X_i$ is a $c_n^{-1}(\textbf{D}_f^b(\textbf{Ab}))$-configuration then $c_n(M(X))$ lies in $\textbf{D}_f^b(\textbf{Ab})$. 
\end{lemma}

\begin{proof}
We prove by induction on the dimension $d$ of the configuration. The statement is obvious for $d=0$. Suppose that the lemma holds for all $c_n^{-1}(\textbf{D}_f^b(\textbf{Ab}))$-configurations of dimension $r<m$. Let $X=X_1 \cup \dots \cup X_n$ be a configuration of dimension $m$. 
By induction hypothesis and localizing distinguished triangle, it is enough to show that $$c_n(M^c(\bigcup_{i=1}^n (X_i\setminus\bigcup_{i \neq j}X_{ij})))= \bigoplus_{i=1}^n c_n(M^c(X_i\setminus\bigcup_{i \neq j}X_{ij}))$$ lies in $\textbf{D}_f^b(\textbf{Ab})$. The later again follows from the localizing triangles 
$$
M^c(\bigcup_{i \neq j}X_{ij}) \rightarrow M^c(X_i) \rightarrow M^c(X_i\setminus\bigcup_{i \neq j}X_{ij}),
$$
and the fact that $c_n$ is a triangulated functor. The statement for configuration of mixed Tate varieties is similar.

\end{proof}

Now let us assume that $Z$ is irreducible. Then by definition there is a map $\pi: \Sigma\to Z$ which satisfies the conditions in Definition \ref{DefNett}.  We claim that $c_n(M(Z))$ lies in $\textbf{D}_f^b(\textbf{Ab})$. We do this by induction on $\dim Z$. We have the following localizing triangles

$$
M^c(\pi^{-1}(Z^\circ)) \to M^c(\Sigma)\to M^c(\cup_{\alpha\neq\alpha_0} \pi^{-1}(\ol Z_\alpha)),
$$
\noindent
corresponding to the inclusion $Z^\circ\hookrightarrow Z$ of the open stratum $Z^\circ=Z_{\alpha_0}$.
By definition of $\hat{Z}$, motivic Leray-Hirsch theorem \ref{Leray Hirsch for Voevodsky motives} and the induction hypothesis $\cup_{\alpha\neq\alpha_0} \pi^{-1}(\ol Z_\alpha)$ is a $c_n^{-1}(\textbf{D}_f^b(\textbf{Ab}))$-configuration. In particular the motivic fundamental invariant
$$
c_n(M(\cup_\alpha \pi^{-1}(\ol Z_\alpha)))
$$ 
lies in $\textbf{D}_f^b(\textbf{Ab})$; see the above lemma \ref{LemMixedArrangement}. As $c_n$ is triangulated, we argue that $c_n(M(\pi^{-1}(Z^\circ)))$ also lie in $\textbf{D}_f^b(\textbf{Ab})$ by generalized gysin triangle. 
Hence we see by theorem \ref{Leray Hirsch for Voevodsky motives} and Definition-Remark \ref{DefRemCn} (d) that 
$$
c_n(M^c(\pi^{-1}(Z^\circ)))=\oplus_{p+q=n}c_p(M^c(Z^\circ))\otimes c_q(M^c(F_{\alpha_0}))
$$ lies in $\textbf{D}_f^b(\textbf{Ab})$. Since $F_{\alpha_0}$ is pure Tate we may argue by proposition \ref{PropDecompPureTate} that $c_n(M(Z^\circ))$ lies in $\textbf{D}_f^b(\textbf{Ab})$. This together with the localizing triangle
$$
M^c(Z^\circ)\to M^c(Z)\to M^c(\cup_\alpha \ol Z_\alpha),
$$
and the fact that $\cup_\alpha \ol Z_\alpha$ is a $c_n^{-1}(\textbf{D}_f^b(\textbf{Ab}))$-configuration, complete our induction assertion for $Z$. Note in addition that, by definition $Z$ is proper and hence there is a canonical isomorphism $M^c(Z)\cong M(Z)$.\\
For non-irreducible $Z$ observe that the irreducible components of $Z$ form a $c_n^{-1}(\textbf{D}_f^b(\textbf{Ab}))$-configuration and therefore we deduce again that $c_n(M(Z))\in \textbf{D}_f^b(\textbf{Ab})$.\\

b) For semi-small resolution $\varrho: \ol\Sigma\to\ol Z$, the decomposition theorem gives the following decomposition of 
$$
Rf_\ast \BQ_\Sigma[n] := \bigoplus_{a\in \CA}
IC_{\ol Z_a}(L_a).
$$
Here $\CA$ denotes the set of admissible strata, i.e. 
$$
\CA := \{\alpha \in \CB ;~ 2 \dim \varrho^{-1}(z) = \dim \ol Z - \dim \ol Z_\alpha, \forall z \in \ol Z_\alpha\}
$$ 
and $L_a$ are the semisimple local systems on $Z_a$ given by the monodromy action on
the maximal dimensional irreducible components of the fibers of $\varrho$ over $\ol Z_\alpha$. The semi-smallness of the resolution $\varrho$, implies that the dimension of every irreducible component of $\Sigma\times_Z\Sigma$ is less than $\dim \Sigma$, which consequently implies that the cycle class map is an isomorphism. From this and Proposition \ref{PropMotivicDecomposition} we deduce that the above decomposition induces a motivic decomposition. Now since $M(\ol Z)$ appears as a summand of $M(\ol \Sigma)$, we conclude that $c_n(M(\ol Z))$ lies in $\textbf{D}_f^b(\textbf{Ab})$.

c)
follows from a) and b) together with the following vanishing observation of Huber and Kahn \cite[Proposition 4.6]{H-K}.

\begin{lemma}\label{Lem_MTVanishingResult}

The motive $M$ in $\textbf{DM}_{gm}^{eff}(k)$ lies in $MT\textbf{DM}_{gm}^{eff}(k)$ if and only if the motivic invariants $c_n(M)$ lie in $\textbf{D}_f^b(\textbf{Ab})$ for all $n$ and $c_n(M)$ vanishes for large enough $n$.

\end{lemma}



d) For this part of the theorem first notice that the bound $\hat{Z}\subseteq \wh\CF\ell_\BP$ is projective, therefore $\hat{Z}$ is algebraizable, see \cite[III,~Thm. 5.4.5]{EGA}, that is it comes by base change from a scheme over $\Spec R_{\hat{Z}}$. By abuse of notation we denote the latter again by $\hat{Z}$. By proper base change theorem there is an isomorphism
$$
\Koh_{\text{\'et}}^q(\hat{Z},\BQ_\ell)\tilde{\to} \Koh_{\text{\'et}}^q(Z,\BQ_\ell),
$$
e.g. see \cite[Ch. VI Cor. 2.7]{Milne1980}. As we observed in c), when $\hat{Z}$ is nett then $M(Z)$ lies in $MT\textbf{DM}_{gm}^{eff}(k)$. So we can conclude by the fact that the abelian category of mixed Tate motives 
over a finite field is semi-simple and its simple objets are the $\BQ(n)$'s. \forget{Thus every
Tate motive is just a direct sum of $\BQ(n)$'s with some shift.}The latter fact follows from the knowledge of the
K-theory of finite fields, which gives the vanishings of the Ext groups $Ext^n(\BQ(a),\BQ(b))=0$ for $n>0$. Recall that $K_{2i}(\BF_q) = 0$ and $$K_{2i-1}(\BF_q) = \BZ/(q^i-1)$$ according to Quillen \cite{Qui}, and therefore they vanish after passing to rational coefficients. \\  

Finally when there is a resolution as in b) then we observed that the motive $M(\ol Z)$ is geometrically mixed Tate. This implies that $M(Z)$ is mixed Tate after a finite extension $L/k$; see Lemma \ref{LemGeoMixedTate}. This shows that some power $F^n$ of Ferobenius is semi-simple, which implies that $F$ is semi-simple.

\end{proof}

\begin{definition}\label{DefRelNett}
Assume that $\CX\to\CY$ admits a stratification $\{\CY_\alpha \}$ such that $\CX_\alpha:= f^{-1}(\CY_\alpha)\to \CY_\alpha$ is a fiber bundle whose fiber $\CS_\alpha$ lies in $\CC_{\ll}$. We say that $\CX\to\CY$ is relatively nett with respect to the stratification $\{\CY_\alpha \}$.

\end{definition}

\begin{lemma}\label{LemRelNett}
Assume that $\CX\to\CY$ is relatively nett with respect to the stratification $\{\CY_\alpha \}$. Then the class of the motive $M^c(\CX)$ inside the Grothendieck ring $K_0[DM_{-}^{eff}(k)]$ lies in the ring generated by the Tate object $\BL$ and $[M^c(\CY_\alpha)]$.
\end{lemma}

\begin{proof}
This follows from the definition \ref{DefRelNett} and theorem \ref{ThmFrobistSemiSimpleOnZ}, gysin triangle and the following result of Gillet and Soul\'e \cite[Proposition 3.2.2.5]{GS}.
\end{proof}

\subsection{Motives and Moduli of $G$-shtukas}\label{SubSectModuliOfG-Sht}

\begin{proposition}
The motive of the stack $\nabla_n^{\CZ}\scrH_D^1(C,\FG)$ lies in the category of geometric motives $DM_{gm}(k)$.
\end{proposition}

\begin{proof}

The stack of bounded $\FG$-shtukas $\CX:=\nabla_n^{\CZ}\scrH_D^1(C,\FG)$ is locally of finite type and Deligne-Mumford. See  \cite[theorem 3.15]{AH_Global} and \cite[theorem 3.1.7]{AH_LM}. Moreover it is seperated by \cite[Theorem 3.15]{AH_Global}. In addition the inertia stack $I(\CX)$ is finite over $\CX$ by \cite[Corollary~3.16]{AH_Global}. Therefore by Keel-Mori Theorem it admits a coarse moduli space $X$ which is separated. See \cite{Conrad}. The motive $M(\CX)$ is naturally isomorphic to $M(X)$.

\end{proof}

Note that the above method, relies on the existence of coarse moduli space, can not be implemented to construct the category of mixed motives over $\nabla_n^{\CZ}\scrH^1(C,\FG)$, as well as the corresponding intersection motives. For this purpose we need some further preliminaries which we recall from \cite{AH_LM}. Moreover as we will describe below, for the sake of simplicity, we restrict our attention to the local case.\\

\noindent
Let us first recall the definition of the category of local $\BP$-shtukas.

\begin{definition}\label{localSht}

Let $\CX$ be the fiber product 
$$
\scrH^1(\Spec \BaseOfD,L^+\BP)\times_{\scrH^1(\Spec \BaseOfD,L\genericG)} \scrH^1(\Spec \BaseOfD,L^+\BP)
$$ 
of groupoids. Let $pr_i$ denote the projection onto the $i$-th factor. We define the groupoid of \emph{local $\BP$-shtukas} $\Sht_{\BP}^{\BD}$ to be 
$$
\Sht_{\BP}^{\BD}\;:=\;\equi\left(\hat{\sigma}\circ pr_1,pr_2\colon \CX \rightrightarrows \scrH^1(\Spec \BaseOfD,L^+\BP)\right)\whtimes_{\Spec\BaseOfD}\Spf\BaseOfD\dbl\zeta\dbr.
$$
where $\hat{\sigma}:=\hat{\sigma}_{\scrH^1(\Spec \BaseOfD,L^+\BP)}$ is the absolute $\BaseOfD$-Frobenius of $\scrH^1(\Spec \BaseOfD,L^+\BP)$. The category $\Sht_{\BP}^{\BD}$ is fibered in groupoids over the category $\Nilp_{\BaseOfD\dbl\zeta\dbr}$ of $\BaseOfD\dbl\zeta\dbr$-schemes on which $\zeta$ is locally nilpotent. We call an object of the category $\Sht_{\BP}^{\BD}(S)$ a \emph{local $\BP$-shtuka over $S$}. More explicitly a local $\BP$-shtuka over $S\in \Nilp_{\BaseOfD\dbl\zeta\dbr}$ is a pair $\ul \CL = (\CL_+,\tauLoc)$ consisting of an $L^+\BP$-torsor $\CL_+$ on $S$ and an isomorphism of the associated loop group torsors $\tauLoc\colon  \hat{\sigma}^\ast \CL \to\CL$. 
\end{definition}

Set $\BP_{\nu_i}:=\FG\times_C\Spec \wh A_{\nu_i}$ and $\hat\BP_{\nu_i}:=\FG\times_C\Spf \wh A_{\nu_i}$. Recall from \cite[Section~5.2]{AH_Local} that to a global $\FG$-shtuka one can associate a tuple of local $\BP$-shtuka at characteristic places. Namely, there is \emph{global-local functor} 

\begin{eqnarray}\label{EqGlobLocFunctor}
\qquad\ul{\wh\Gamma}\;:=\; \prod_i\wh\Gamma_{\nu_i}\colon \es \nabla_n\scrH^1(C,\FG)^{\ul\nu}(S) & \longto & \prod_i \Sht_{\BP_{\nu_i}}^{\Spec \wh A_{\nu_i}}(S)\, ,\label{G-LFunc}
\end{eqnarray}
This mirrors the assignment of a $p$-divisible group to an abelian variety over $\BF_q$ (this more generally mirrors the crystalline realization of a motive).  One may use this assignment to impose boundedness conditions to the moduli of $\FG$-shtukas.

\begin{definition-remark}\label{DefBCOnGShtAndLevelStr}

\begin{enumerate}
\item
Let $\hat{Z}$ be a bound with reflex ring $R_{\hat{Z}}$. Let $\CL_+$ and $\CL_+'$ be $L^+\BP$-torsors over a scheme $S$ in $\Nilp_{R_{\hat{Z}}}$ and let $\delta\colon \CL\isoto\CL'$ be an isomorphism of the associated $L\genericG$-torsors. We consider an \'etale covering $S'\to S$ over which trivializations $\alpha\colon\CL_+\isoto(L^+\BP)_{S'}$ and $\alpha'\colon\CL'_+\isoto(L^+\BP)_{S'}$ exist. Then the automorphism $\alpha'\circ\delta\circ\alpha^{-1}$ of $(L\genericG)_{S'}$ corresponds to a morphism $S'\to L\genericG\whtimes_\BaseOfD\Spf R_{\hat{Z}}$. We say that $\delta$ satisfies \emph{local boundedness condition} (\textbf{LBC}) by $\hat{Z}$ if for any such trivialization and for all finite extensions $R$ of $\BaseOfD\dbl\zeta\dbr$ over which a representative $\hat{Z}_R$ of $\hat{Z}$ exists, the induced morphism $S'\whtimes_{R_{\hat{Z}}}\Spf R\to L\genericG\whtimes_\BaseOfD\Spf R\to \wh{\SpaceFl}_{\BP,R}$ factors through $\hat{Z}_R$. Furthermore we say that a local $\BP$-shtuka $(\CL_+, \tauLoc)$ is \emph{bounded by $\hat{Z}$} if the isomorphism $\tauLoc^{-1}$ satisfies LBC by $\hat{Z}$.\\

\item

Fix an $n$-tuple $\ul\nu=(\nu_i)$ of places on the curve $C$ with $\nu_i\ne\nu_j$ for $i\ne j$. Let $\ul{\hat{Z}}:=(\hat{Z}_i)_i$ be an $n$-tuple of bounds in the above sense and set $R_\ul{\hat{Z}}:=R_{\hat{Z}_1}\hat{\otimes}_{\BF_q}\dots \hat{\otimes}_{\BF_q} R_{\hat{Z}_n}$. We say that a tuple $(\CG,\CG',\ul s,\phi)$ in $Hecke_n(C,\FG)^\ul\nu\times_{\wh A_\ul\nu}\Spf R_\ul{\hat{Z}}$ is bounded by $\ul{\hat{Z}}$ if for each $i$ the inverse $\wh\phi_{\nu_i}^{-1}$ of the associated isomorphism $\wh\phi_{\nu_i}:=\L_{\nu_i}(\phi_i):\L_{\nu_i}\CG'\to \L_{\nu_i}\CG$ of $LP_{\nu_i}$-torsors at $\nu_i$ satisfies LBC by $\ul{\hat{Z}}$ in the above sense. We denote the resulting formal stack by $Hecke_n^{\ul{\hat{Z}}}(C,\FG)$, and sometimes we abbreviate this notation by $Hecke_n^{\ul{\hat{Z}}}$.\\
Similarly we say that a $\FG$-shtuka $\ul\CG$ in $\nabla_n\scrH^1(C,\FG)^{\ul\nu}(S)$ is bounded by $\ul{\hat{Z}}$ if for every $i$ the associated local $\BP_{\nu_i}$-shtuka $\ul\CL_i$ is bounded by $\hat{Z}_i$. Here we set $(\ul\CL_i):=\ul{\wh\Gamma}(\ul\CG)$. We denote the moduli stack obtained by imposing these boundedness conditions by $\nabla_n^{\ul{\hat{Z}}}\scrH^1_D(C,\FG)$. 
\item
Furthermore, using tannakian formalism one can equip this moduli stack with $H$-level structure, for a compact open subgroup $H\subset \FG (\BA_Q^{\ul\nu})$. Here $\BA_Q^{\ul \nu}$ is the ring of adeles of $C$ outside the fixed $n$-tuple $\ul\nu:=(\nu_i)_i$ of places $\nu_i$ on $C$. For detailed account on $H$-level structures on a global $\FG$-shtuka, we refer the reader to \cite[Chapter~6]{AH_Global}. We denote the resulting moduli stack by $\nabla_n^{H, \ul{\hat{Z}}}\scrH^1(C,\FG)$.

\item
For $(\CG,\CG',\ul s,\tau)$ in $Hecke_n(C,\FG)$, one can control the relative position of $\CG$ and $\CG'$ under $\tau^{-1}$ also by means of global boundedness condition $\CZ$, see \cite[Definition~3.1.3]{AH_LM}. This leads to definition of the algebraic stack $Hecke_{n,D}^\CZ(C,\FG)$ and $\nabla_n^\CZ\scrH_D^1(C,\FG)$. We sometimes use the subscript in our notation $\alpha$ $\nabla_n^\CZ\scrH_D^1(C,\FG)$ (resp. $Hecke_{n,D}^\CZ(C,\FG)$) to denote the corresponding moduli stacks obtained by restricting the above constructions to the open substack $\scrH_\alpha^1$ of the stack $\scrH^1(C,\FG)$; see definition-remark \ref{Def-RemH1}.

\end{enumerate}
\end{definition-remark}

As a significant feature of the assignment \ref{EqGlobLocFunctor} one can prove that the deformation theory of a global $\FG$-shtuka can be read of the associated local $\BP$-shtukas via $\wh{\ul\Gamma}$. Let us explain it a bit more explicitly. Let $S \in \Nilp_{\wh A_\ul\nu}$ and let $j: \ol S \to S$ be a closed subscheme defined by a locally nilpotent sheaf of ideals $\CI$. Let $\ol{\ul\CG}$ be a global $\FG$-shtuka $\nabla_n\scrH^1(C,\FG)^{\ul \nu}(\bar{S})$. We let $Defo_S(\bar{\ul{\cG}})$ denote the category of infinitesimal deformations of $\ol{\ul\CG}$ over $S$. Similarly for a local $\BP$-shtuka $\bar{\CL}$ in $\Sht_\BP^\BD(S)$ we define the category of lifts  $Defo_S(\bar{\ul\CL})$ of $\bar{\ul\CL}$ to $S$. Then one can prove that Let $\bar{\ul{\cG}}:=(\bar{\CG},\bar{\tau})$ be a global $\FG$-shtuka in $\nabla_n\scrH^1(C,\FG)^{\ul \nu}(\bar{S})$. Then the functor 
$$
Defo_S(\bar{\ul{\cG}})\longto \prod_i Defo_S(\ul{\bar\CL}_i)\,,\quad \bigl(\ul\CG,\alpha)\longmapsto(\ul{\wh\Gamma}(\ul\CG),\ul{\wh\Gamma}(\alpha)\bigr)
$$ 
induced by the global-local functor (\ref{G-LFunc}), is an equivalence of categories. Here $(\ul{\bar\CL}_i)_i$ denote the tuple $\wh{\ul\Gamma}(\bar{\ul{\cG}})$. This mirrors the Serre-Tate's theorem for abelian varieties. Based on this observation one can proceed by proving the following result

\begin{theorem}\label{PropLocalModelHecke}

Fix an n-tuple of bounds $\ul{\hat{Z}}:=(\hat{Z}_i)_i$ and let $\hat{Z}_{i,R_{\nu_i}}$ be a representative of $\hat{Z}_i$ over $R_{\nu_i}$. Set $R_{\hat{Z}_\ul \nu} :=R_{\hat{Z}_1}\hat{\otimes}_{\BF_q}\dots \hat{\otimes}_{\BF_q} R_{\hat{Z}_n}$ and $R_\ul \nu :=R_{\nu_1}\hat{\otimes}_{\BF_q}\dots \hat{\otimes}_{\BF_q}R_{\nu_n}$. We have the following statements

\begin{enumerate}
\item
There is a formal algebraic stack $\wt{Hecke}_{R_{\ul\nu}}^{\ul{\hat{Z}}}$ and roof of morphisms 
\begin{equation}\label{HeckeRoof}
\xygraph{
!{<0cm,0cm>;<1cm,0cm>:<0cm,1cm>::}
!{(0,0) }*+{\wt{Hecke}_{R_{\ul\nu}}^{\ul{\hat{Z}}}}="a"
!{(-2.5,-2.5) }*+{{Hecke_n^\ul{\hat{Z}}}\times_{R_{\ul{\hat{Z}}}}R_{\ul\nu}}="b"
!{(2.5,-2.5) }*+{\prod_i \hat{Z}_{i,R_{\nu_i}}.}="c"
"a":^{\pi}"b" "a":_{f}"c"
}
\end{equation}  
\noindent
Furthermore, in the above roof, the formal stack $\wt{Hecke}_{R_{\ul\nu}}^{\ul{\hat{Z}}}$ is an $\prod_i L^+\BP_{\nu_i}$-torsor over $Hecke_{R_\ul\nu}^{\ul{\hat{Z}}}:={Hecke_n^\ul{\hat{Z}}}\times_{R_{\ul{\hat{Z}}}}R_{\ul\nu}$ under the projection $\pi$. Moreover for a geometric point $y$ of $Hecke_{R_\ul\nu}^{\ul{\hat{Z}}}$, the $\prod_i L^+\BP_{\nu_i}$-torsor $\pi:\wt{Hecke}_{R_\ul\nu}^{\ul{\hat{Z}}}\to Hecke_{R_\ul\nu}^{\ul{\hat{Z}}}$ admits a section $s$, over an \'etale neighborhood of $y$, such that the composition $f\circ s$ is formally smooth. 

\item Consider the formal algebraic stack $\nabla_n^{H, \ul{\hat{Z}}}\scrH^1$ of $\FG$-shtukas bounded by $\ul{\hat{Z}}$ and endowed with level $H$-structure, for compact open subgroup $H\subseteq \FG(\BA_C^\ul\nu)$. The above roof of morphisms induces the following 

\begin{equation}\label{eqnablaHRoof} 
\xygraph{
!{<0cm,0cm>;<1cm,0cm>:<0cm,1cm>::}
!{(0,0) }*+{\nabla_n^{H, \ul{\hat{Z}}}\wt{\scrH_{R_\ul\nu}^1}}="a"
!{(-1.5,-1.5) }*+{\nabla_n^{H, \ul{\hat{Z}}}\scrH_{R_\ul\nu}^1}="b"
!{(1.5,-1.5) }*+{\prod_i \hat{Z}_{i,R_{\nu_i}},}="c"
"a":^{\pi'}"b" "a":_{f'}"c"
}  
\end{equation}

and that the local section $s$ induces a local section $s'$ such that $s' \circ f'$ is formally \'etale.

\end{enumerate}

\end{theorem}

\begin{remark}\label{RemStratification}
The above roof of morphisms \ref{eqnablaHRoof} in particular induces the following smooth morphism
$$
\nabla_n^{H, \ul{\hat{Z}}}\scrH_{R_\ul\nu}^1\to \prod_i  L^+\BP_{\nu_i} \backslash \hat{Z}_{i,R_{\nu_i}}
$$
of formal algebraic stacks. Thus we obtain a natural stratification on the special fiber $\nabla_n^{H, \ul{\hat{Z}}}\scrH_{R_\ul\nu,s}^1$ which is induced by the orbits of the $\prod_i L^+\BP_{\nu_i}$-action on $\prod_i Z_i$ and their incidence relation, where $Z_i$ denotes the special fiber of $\hat{Z}_i$.

\end{remark}

Recall that the ind-algebraic structure on $Hecke_n(C,\FG)$ induces an ind-algebraic structure on $GR_n:=GR_n(C,\FG):=\dirlim GR_n^{(\rho,\ul\omega)}(C,\FG)$; see definition-remark \ref{Def-RemHecke}(c). This allows to define the derived category of motives $DM(GR_n)$ as the colimit of $DM(GR_n^\ul\omega)$, where the latter can be defined in the sense of \cite{CD}.

\begin{proposition}
The $D$-level structure can be taken enough large such that $\nabla_n^{(\rho,\ul\omega)}\scrH_D^1(C,\FG)_\alpha$ (resp. $\nabla_n^\CZ\scrH_D^1(C,\FG)_\alpha$) becomes representable by a quasi-projective scheme. Consequently the level $H$-structure can be taken enough small such that $\nabla_n^{\ul{\hat{Z}}, H}\scrH^1(C,\FG)_\alpha$ becomes representable by a quasi-projective formal scheme over $\Spf \wh A_\ul\nu$.
\end{proposition}

\begin{proof}

Fix a faithful representation $\rho\colon\FG\into\SL(\CV_0)$  for some vector bundle $\CV_0$ of rank $r$, with quasi-affine quotient $\SL(\CV_0)/\FG$; see \cite[Proposition 2.2.b)]{AH_Global}. 
First observe that $\scrH^1(C,\GL(\CV_0))$ can be identified with the stack $Vect_C^r$ of vector bundles of rank $r=\rk \CV_0$ (and hence with $\scrH^1(C,\GL_r)$). The stack $\scrH^1(C,\SL(\CV_0))$ is determined by requiring that the determinant of the $\FG$-bundles $\CG\in\scrH^1(C,\GL(\CV_0))$ are trivial. In other words there is the following Cartesian diagram of algebraic stacks

$$
\CD
\scrH^1(C,\GL(\CV_0))@>{\det}>>\scrH^1(C,\BG_m)\\
@AAA @A{tr}AA\\
\scrH^1(C,\SL(\CV_0))@>>>\Spec \BF_q,
\endCD
$$

where $\det$ is induced by $\det: \GL_n\to\BG_m$ and the right vertical arrow corresponds to the trivial line bundle. It is well-known that the stack of vector bundles $Vect_C^r$ admits a covering by open substacks of finite type $\BF_q$ which further become representable by a quasi- projective scheme after endowing with $D$-level structure, for large enough $D$. As the morphism $tr:\Spec \BF_q\to \scrH^1(C,\BG_m)$ is schematic and quasi-projective, we see that the same holds for $\scrH^1(C,\SL(\CV_0))\to \scrH^1(C,\GL(\CV_0))$. Furthermore as $\SL(\CV_0)/\FG$ is quasi-affine, therefore the morphism
$$
\scrH^1(C,\FG)\to \scrH^1(C,\SL(\CV_0))
$$
is quasi-projective. We argue that the divisor $D$ can be chosen enough large such that $\scrH_D^1(C,\FG)_\alpha$ becomes representable by a quasi-projective scheme. By Definition-Remark \ref{Def-RemHecke}.b) This implies that for such $D$-level structure $Hecke_D^{(\rho,\ul\omega)}(C,\FG)_\alpha$ is quasi-projective, which further implies that $\nabla_n^{(\rho,\ul\omega)}\scrH_D^1(C,\FG)_\alpha$ is quasi-projective, as it is defined to be the pull-back of the graph of Frobenius $\Gamma_\sigma \subseteq \scrH^1(C,\FG)_\alpha\times_{\BF_q}\scrH^1(C,\FG)_\alpha$ under
$$
Hecke_D^{(\rho,\ul\omega)}(C,\FG)_\alpha\to \scrH^1(C,\FG)_\alpha\times \scrH^1(C,\FG)_\alpha.
$$ 
To see the statement for the formal stack $\nabla_n^{\ul{\hat{Z}}, H}\scrH^1(C,\FG)$, take $$H:=\ker(\FG(\BA^\ul\nu\to \FG(\CO_D))),$$ for relevant $D$ as above. Using tannakian formalism one can observe that 
$$
\nabla_n^{\ul{\hat{Z}}, H}\scrH^1(C,\FG)\cong \nabla_n^{\ul{\hat{Z}}}\scrH_D^1(C,\FG),
$$ 
see \cite[Theorem 5.2.5]{Ara}. Note that $\nabla_n\scrH_D^1(C,\FG)=\dirlim[\ul\omega] \nabla_n^{(\rho,\ul\omega)}\scrH_D^1(C,\FG)$, and thus the morphism 
$$
\nabla_n^{\ul{\hat{Z}}}\scrH_D^1(C,\FG)\to \nabla_n\scrH_D^1(C,\FG)
$$ 
factors through a closed immersion to $\nabla_n^{(\rho,\ul\omega)}\scrH_D^1(C,\FG)$.  

\end{proof}

According to the above proposition one can define the category $DM(\nabla_n^{(\rho,\ul\omega)}\scrH_D^1(C,\FG))$, for large enough $D$, and consequently $DM(\nabla_n^{(\rho,\ul\omega)}\scrH^1(C,\FG))$ as $DM(\CC^\bullet)$, where $\CC^\bullet$ denotes the Cech simplicial complex associated to the \'etale cover
$$\nabla_n^{(\rho,\ul\omega)}\scrH_D^1(C,\FG)\to\nabla_n^{(\rho,\ul\omega)}\scrH^1(C,\FG).
$$

Recall from Definition-Remark \ref{Def-RemNblaH} (b) that the moduli stack of global $\FG$ shtukas $\CX:=\nabla_n\scrH^1(C,\FG)$ admits a Deligne-Mumford ind-algebraic structure, i.e. it can be viewed as the direct limit 
$$
\CX:=\dirlim \CX^{(\rho,\ul\omega)}
$$
of closed immersions of Deligne-Mumford stacks $\CX^{(\rho,\ul\omega)}:=\nabla_n^{(\rho,\ul\omega)} \scrH^1(C,\FG)$. 
Thus one can consider the category $DM(\nabla_n\scrH^1(C,\FG))$ of motives over $\nabla_n\scrH^1(C,\FG)$ as the colimit of the motivic categories $DM(\CX^{(\rho,\ul\omega)})$ under $i_{\ul\omega\ul\omega'\ast}$, for closed immersions $i_{\ul\omega\ul\omega'}:\CX^{(\rho,\ul\omega)}\to \CX^{(\rho,\ul\omega')}$. 

For the sake of simplicity let us now restrict our attention to the special fiber of $\nabla_n^{\wh Z,H}\scrH_s^1:=\nabla_n^{\wh Z,H}\scrH^1(C,\FG)_s$ above $\ul\nu$. Consider the motivic category $DM(\nabla_n^{\wh Z,H}\scrH_s^1)$. As we mentioned in remark \ref{RemStratification}, according to the local model diagram, we obtain a stratification on $\nabla_n^{\wh Z,H}\scrH_s^1$. Let $\mathds{U}:=(\nabla_n^{\wh Z,H}\scrH^1)^\circ$ denote the open stratum. Following \cite[Section 2]{Wil} we define the motivic intersection complex $ICM(\nabla_n^{\wh Z,H}\scrH_s^1)$ as the intermediate extension $j_{\ast!}\mathds{1_U}$, where $j$ denotes the open immersion $\mathds{U}\hookrightarrow\nabla_n^{\wh Z,H}\scrH_s^1$. We refer to \cite[Section 4]{Wil} for the existence problem.

\begin{proposition}\label{Prop_MotIntersectionCompNablaScrH}
The motivic intersection complex $ICM(\nabla_n^{H, \wh Z_\ul\nu}\scrH_s^1)$ and the restriction of $ICM(Hecke_{n,s}^{\CZ_\ul\nu})$ to $\nabla_n^{H, \wh Z_\ul\nu}\scrH_s^1$ agree up to some shift and Tate twist.
\end{proposition}

\begin{proof}
The proof for intersection complexes is given in \cite[Proposition 4.5.2]{AH_LM}. Note that this proof only relies on the Grothendieck's six-functor formalism and its basic properties and hence similar arguments can be used to establish the proof in the motivic context. 
\end{proof}

\begin{remark}
Recall from \ref{Def-RemHecke} that the stack $Hecke_n(C,\FG)$ and $GR_n\times C^n$ are locally isomprphic for the \'etale topology on $C^n\times \scrH^1(C,\FG)$ and further this is compatible with boundedness condition and by definition preserves the stratification. The above also suggests to consider the following assignment 
$$
\CZ \mapsto ICM_\CZ:=i^! ICM(Hecke_n^\CZ(C,\FG))\in DM(\nabla_n\scrH^1(C,\FG)).
$$

\noindent
This assignment might be used to transform Satake classes from $DM(GR_n)$ to $DM(\nabla_n\scrH^1(C,\FG))$ in global situation and for (non-constant) ramified cases.
\end{remark}

\begin{question} (Motivic invariants of $\nabla_n^{\ul{\hat{Z}}}\scrH_s^1$) Recall from the discussions in subsection \ref{SubSectSSFrobLM} that the motivic invariant $c_n(X)$ measures the graded pieces of the chunks of slice filtration.  
Regarding our previous discussions one may ask for possible descriptions of the motivic invariants $c_n(M(\nabla_n^{H, \wh Z_\ul\nu}\scrH_s^1))$ associated with the motive of the special fiber $\nabla_n^{H, \ul{\hat{Z}}}\scrH_s^1$. Similar question maybe posed in the global situation and also for $c_n(\pi_! ICM(\nabla_n^{H, \wh Z_\ul\nu}\scrH_s^1))$, where $\pi:\nabla_n^{H, \wh Z_\ul\nu}\scrH_s^1\to \Spec k$ is the structure map. When $G$ is split reductive and $C=\BP^1$, it seems plausible to expect that these fundamental invariants lie in $\textbf{D}_f^b(\textbf{Ab})$. 
\end{question}

\subsection{The Motive Of $\CZ$ and $Hecke_n^\CZ$}\label{SubSecLocToGlob}

In this subsection we assume that $\FG$ is parahoric.

\begin{proposition}\label{PropLocToGlob}
There is an assignment 

$$
\gamma: \{\text{global boundedness conditions}~ \CZ\} \to \{\text{n-tuple of local boundedness conditions}~ \ul{\hat{Z}}\}.
$$

Furthermore a tuple $\ul{\hat{Z}}$ can give rise to a global boundedness condition $\CZ$ with $\gamma(\CZ)=\ul{\hat{Z}}$. Moreover if for every $i$ the motive of the generic fiber of $\ul{\hat{Z}}$ lies in the category of the pure Tate motives then all fiber of the corresponding boundedness condition $\CZ$ (as a family over $C^n$) are geometrically pure Tate.

\end{proposition}

\begin{proof}

Recall from \cite[Prop 4.3.3]{AH_LM} that to a global boundedness condition $\CZ$ one can assign a tuple of local boundes $\ul{\hat{Z}}$. This is done by taking formal completion at the characteristic places $\ul\nu$. Vise versa, starting from a tuple $\ul{\hat{Z}}:=([\hat{Z}_i])_i$ of local bounds we can perform a global boundedness condition $\CZ$ in the following way. Let $\hat{Z}_i$ be a representative of $[\hat{Z}_i]$ over the ring $R_i:=R_{\hat{Z}_i}$. Consider the corresponding finite field extension $\wh{Q}_{\hat{Z}_i}/\wh{Q}_{\nu_i}$. It comes from a global field extension $\wt{Q}_i/Q$.  Adjoining these global fields for all $i$, we obtain an extension $\wt{Q}/Q$. Let $\wt{C}$ be the curve corresponding to the field extension $\wt{Q}/Q$. Note that as $\hat{Z}_i$ is projective, it admits a model over $\Spec R_i$, which by abuse of notation we still denote by $\hat{Z}_i$. The generic fiber of $GR_1(\wt{C},\FG)$ is the usual affine Grassmannian $Gr_{G_\wt Q}$. Note that $GR_1(\wt{C},\FG)$ and $GR_1(C,\FG)\times_{C}\wt C$ are isomorphic on the  locus where $\wt C\to C$ is \'etale. Note further that we have $\CF\ell_{\FG_\nu,R_i}=GR_1(\wt C, \FG)\times_{\wt C} R_i$. The generic fiber of $\hat{Z}_i$ defines a closed subscheme $Z_{i,\eta}$ of $Gr_{G_\wt Q}$.  Now consider the fiber product 
$$
Z:=Z_{1,\eta}\times \cdots \times Z_{n,\eta} \subseteq Gr_{G_\wt Q}^n=(GR_n(C, \FG)\times_{C^n}\wt C^n)_{\eta_{\wt C}^n}
$$
and let $\CZ$ be the Zariski closure of $Z$ in $GR_n(C, \FG)\times_{C^n} \wt{C}^n$. This defines a global boundedness condition $\CZ$ corresponding to the local boundedness condition $\ul{\hat{Z}}$. \\

The second part follows from the degeneration method. If $n=1$ then the fibers $F$ of $\CZ$ over an open subscheme $U\subseteq \wt C$ are pure Tate. Recall that this equivalently means that the diagonal $\Delta_F$ of $F$ fully decomposes in $Ch(F\times F)$. By \cite[Th\'eor\`eme 2.3]{CoPi} we see that the fibers over $\wt C\setminus U$ are also pure Tate. For $n=2$, a general fiber over $\ul x:=(x_1,x_2)$ outside the diagonal $\diag\subseteq \wt C^2$ is pure Tate by construction. For $\ul x=(x_1\times x_1)$ consider the restriction $\CZ|{x_1\times \wt C}$. Then using the fact that the general fiber is pure Tate, and that it specializes as we explained above, we argue that $M(F)$ is pure Tate in general. One can proceed similarly to prove the assertion for $n>2$.

\end{proof}

 Let $\ul\mu:=\{\mu_i\}_{i=1,\dots, n}$ be a set of geometric conjugacy classes of cocharacters in $G$ which are defined over a finite
separable extension $E/Q$. We say that a global boundedness condition $\CZ$ is \emph{generically defined by} $\{\mu_i\}$  if it arises in the following way. Namely,  each $\mu_i$ defines a closed subscheme of $GR_1\times_C \wt{C}_E$ which we denote by $GR_{\preceq\mu_i}$. Note that $GR_{\preceq\mu_i}$ is proper flat with geometrically connected equi-dimensional fibers over $\wt{C}_E$. Now consider the bound $\CZ:=\CZ(\ul\mu)$ which is given by the class of the Zariski closure in $GR_n\times_{C^n}\wt{C}_E^n$ of the restriction of the fiber product $GR_{\preceq\mu_1}\times \dots \times GR_{\preceq\mu_n}$ of global affine Schubert varieties $GR_{\preceq\mu_i}$ to the complement of the big diagonal in $\wt{C}_E^n$.

When $\FG$ is constant $G\times_{\BF_q} C$ for split reductive group $G$ over $\BF_q$, then we have $\wt C=C$ and we consider the obvious action of the symmetric group $S_n$ on $C^n$. This induces a stratification $(C^n)_\ul\alpha$. Here $\ul\alpha$ denote a subset of $\ul n=\{1,\cdots,n\}$. One can see the following

\begin{corollary}\label{CorMotiveOfGr}

Keep the above notation. We have the following statements

\begin{enumerate}

\item
Assume that $\FG$ is parahoric and $S(\mu_i)$ are smooth and pure Tate. The motive $M(\CZ)$ of $\CZ:=\CZ(\ul\mu)$ over $\wt C^n$ lies in the thick subcategory of $DM_{gm}(\ol k)$ generated by pure Tate motives and $M(\wt C)$. In particular when $\wt C=C=\BP^1$ then the motive $M(\CZ)$ in $DM_{gm}(k)$ is geometrically mixed Tate. \\

\noindent
Furthermore when $\FG$ is constant we have

\item

The class of the motive $[M(\CZ)]$ in the Grothendieck ring $K_0[DM_{-}^{eff}(k)]$ can be written as the following sum

$$
[M^c(\CZ)]= \sum_{\ul\alpha} [M^c((C^n)_\ul\alpha)]\cdot[M(S(\mu_\ul\alpha)]\cdot \prod_{i\notin\ul\alpha} [M(S(\mu_i))].
$$

Here $\ul\alpha$ runs over subsets of $\{1,\dots,n\}$ and
$\mu_\ul\alpha:=\sum_{\alpha\in\ul\alpha}\mu_\alpha$\forget{, $S^\ul\alpha:=\prod_{\alpha\notin\ul\alpha} S(\mu_\alpha)$}. In particular when $C=\BP^1$ it lies in the subring generated by $MT\textbf{DM}_{gm}^{eff}(k)$.

\item
Assuming \cite[conjecture 3.4]{Beh07}, the class of the motive of $Hecke_n^{\CZ}(C,\FG)$ can be expressed in the following way

$$
[M^c(Hecke_n^{\CZ}(C,\FG)]=~~~~~~~~~~~~~~~~~~~~~~~~~~~~~~~~~~~~~~~~~~~~~~~~~~~~~~~~~~~~~~~~~~~~~~~~~~~~~~~~~~
$$
$$ 
~~~~~~~|\pi_1(G)|\BL^{(g-1)\dim G} \prod_{i=1}^r
Z(C,\BL^{-d_i})\cdot\sum_{\ul\alpha} [M^c((C^n)_\ul\alpha)]\cdot[M(S(\mu_\ul\alpha)]\cdot \prod_{i\notin\ul\alpha} [M(S(\mu_i))].
$$

\forget{
$$
[M^c(Hecke_n^{\CZ}(C,\FG)]= |\pi_1(G)|\BL^{(g-1)\dim G} \prod_{i=1}^r
Z(C,\BL^{-d_i})\cdot\sum_{\ul\alpha} [M((C^n)_\ul\alpha)]\cdot[M(S_\ul\alpha)]\cdot [M(S^\ul\alpha)].
$$
}

Here $\BL$ denotes the class corresponding to $\BA^1$, $g$ denotes the genus  of $C$, $Z(C,t)$ denotes the motivic zeta function of $C$, and $d_i$'s are given by the class $[G] = \BL^{\dim G} \prod_{i=1}^r (1 - \BL^{-d_i})$ of $G$ in the Grothendieck ring.

\end{enumerate}
\end{corollary}

\begin{proof}
a) Using simple induction argument, this follows from the construction of $\CZ$ and the above proposition \ref{PropLocToGlob}. Recall that regarding the construction of $\CZ$, the restriction of $\CZ$ to the stratum corresponding to $\ul\alpha$ is a fiber bundle, whose fiber equals $S_\ul\alpha\times S^\ul\alpha:=S(\mu_\ul\alpha)\times\prod_{i\notin\ul\alpha}S(\mu_i)$. 
 When $C=\BP^1$ the statement follows from example \ref{ExampleBigDiagonal}, theorem \ref{ThmFrobistSemiSimpleOnZ}, K\"unneth formula and properness of $S^\ul\alpha$ and $S_\ul\alpha$. b) First part follows from construction and properness of $\CZ$. For $C=\BP^1$ the statement follows from part (a). c) This follows from \cite[Proposition 2.0.11]{AH_LM} and the above part b).

\end{proof}

\begin{remark}
Recall that for a minuscule coweight $\mu$ the Schubert variety $S(\mu)$ is cellular and thus the motive $M(S(\mu))$ is pure Tate. The above corollary provides a proof for the following fact that for a tuple $\ul\mu:=(\mu_i)$ of minuscule coweights the motive of $S(\ul\mu)$ in $DM_{gm}(\ol k)$ is pure Tate. Note that for this we do not need to consider the fibers of the resolution of singularities of $S(\ul\mu)$.
\end{remark}

\Verkuerzung

\vfill

\bigskip

\noindent
\begin{minipage}[t]{0.9\linewidth}
\noindent
Esmail Arasteh Rad, Universit\"at M\"unster, 
Mathematisches Institut, Einsteinstr. 62, D-48149 M\"unster, Germany\\
erad@uni-muenster.de\\
\\[1mm]
\end{minipage}

\noindent
\begin{minipage}[t]{0.9\linewidth}
\noindent
Somayeh Habibi, School of Mathematics, Institute for Research in Fundamental Sciences (IPM), P.O.Box: 19395-5746, Tehran, Iran \\ 
shabibi@ipm.ir
\end{minipage}

\end{document}